\documentclass[requno]{amsart}
\usepackage{amsmath,amsthm,amssymb}

\makeatletter
    
    \@addtoreset{equation}{section}
  \makeatother

\newtheorem{definition}{Definition}[section]
\newtheorem{proposition}[definition]{Proposition}
\newtheorem{theorem}[definition]{Theorem}
\newtheorem{lemma}[definition]{Lemma}
\newtheorem{corollary}[definition]{Corollary}
\newtheorem{remark}{Remark}[section]

\newcommand{\ep}{\varepsilon}

\newcommand{\R}{\mathbb{R}}

\newcommand{\N}{\mathbb{N}}
\newcommand{\pa}{\partial}
\newcommand{\lr}[1]{\langle{#1}\rangle}

\pagebreak

\title[Wave equation with growing damping]{
Diffusion phenomena for the wave equation
with space-dependent damping term growing at infinity
}
\author{Motohiro Sobajima}
\address[M. Sobajima]{Department of Mathematics, 
Faculty of Science and Technology, Tokyo University of Science,  
2641 Yamazaki, Noda-shi, Chiba, 278-8510, Japan}
\email{msobajima1984@gmail.com}

\author{Yuta Wakasugi}
\address[Y. Wakasugi]{
Department of Engineering for Production and Environment,
Graduate School of Science and Engineering,
Ehime University,
3 Bunkyo-cho, Matsuyama, Ehime, 790-8577, Japan}
\email{wakasugi.yuta.vi@ehime-u.ac.jp}

\begin{document}
\begin{abstract}
In this paper, we study the asymptotic behavior of solutions
to the wave equation with
damping depending on the space variable and growing at the spatial infinity.
We prove that the solution is approximated
by that of the corresponding heat equation as time tends to infinity.
The proof is based on semigroup estimates for
the corresponding heat equation
and weighted energy estimates for the damped wave equation.
To construct a suitable weight function for the energy estimates,
we study a certain elliptic problem.
\end{abstract}
\keywords{Damped wave equation; diffusion phenomena;
unbounded damping}

\maketitle
\section{Introduction}
\footnote[0]{2010 Mathematics Subject Classification. 35L20; 35B40; 47B25}

Let $N\ge 2$ and
let $\Omega$ be an exterior domain in $\mathbb{R}^N$
with a smooth boundary
$\partial \Omega$.
We assume that
$0 \notin \bar{\Omega}$
without loss of generality.
We study the initial-boundary value problem for
the damped wave equation
\begin{align}%
\label{dw}
	\left\{\begin{array}{ll}
	u_{tt} - \Delta u + a(x) u_t = 0,&x\in \Omega, t>0,\\
	u(x,t) = 0,&x\in \partial \Omega, t>0,\\
	u(x,0) = u_0(x), u_t(x,0)=u_1(x),&x\in \Omega.
	\end{array}\right.
\end{align}%
Here
$u = u(x,t)$ is a real-valued unknown function.
We assume that
$a(x) \in C^2(\bar{\Omega})$,
$a(x) > 0$ on $\bar{\Omega}$
and, there exist constants
$\alpha > 0$ and $a_0 > 0$
such that
\begin{align}%
\label{a}
	\lim_{|x|\to \infty} |x|^{-\alpha}a(x) = a_0,
\end{align}%
that is,
$a(x)$ is unbounded and diverges at the spatial infinity.
The initial data
$(u_0, u_1)$
is compactly supported, let us say
${\rm supp\,} (u_0,u_1) \subset B(0,R_0)$
with some $R_0 > 0$,
and
satisfies the compatibility condition of order $1$, namely,
\begin{align}%
\label{ini}
	(u_0, u_1) \in (H^2(\Omega) \cap H^1_0(\Omega) ) \times H^1_0(\Omega).
\end{align}%
Noting the compactness of the support of the initial data
and the finite propagation property,
we can show that 
there exists a unique strong solution
\begin{align}%
\label{sol}
	u \in \bigcap_{i=0}^2 C^i([0,\infty);H^{2-i}(\Omega))
\end{align}%
in a standard way (see Ikawa \cite[Theorem 2]{Ika68}).

The term
$a(x)u_t$
describes the damping effect,
which plays a role in reducing the energy of the wave.
Our aim is to clarify
how the strength of the damping effects
the behavior of the solution.
Particularly, in this paper we study
the case where the strength of the damping increases at
the spatial infinity,
and as a typical example,
we consider the damping satisfying \eqref{a}.
We investigate
the asymptotic profile of the solution
and how the exponent
$\alpha$
is related to the decay rate of the energy of the solution.

For this purpose,
we also consider the corresponding parabolic problem
\begin{align}%
\label{h}
	\left\{\begin{array}{ll}
	v_t - a(x)^{-1}\Delta v = 0,&x\in \Omega, t>0,\\
	v(x,t) = 0,&x\in \partial \Omega, t>0,\\
	v(x,0) = v_0(x),&x\in \Omega.
	\end{array}\right.
\end{align}%
This equation is formally obtained
by dropping the term $u_{tt}$ from the equation \eqref{dw}
and dividing it by $a(x)$.

To observe the relation between the equations \eqref{dw} and \eqref{h},
we formally consider the case
$a(x) = |x|^{\alpha}$
and
$\Omega = \mathbb{R}^N$,
namely, we drop the boundary condition and consider
an initial value problem.
For the equation \eqref{dw}, we change the variables as
\begin{align*}%
	u(x,t) = \phi ( \lambda^{1/(2+\alpha)} x, \lambda^{1/2}t ),\quad
	y = \lambda^{1/(2+\alpha)} x,\ s = \lambda^{1/2}t
\end{align*}% 
with a parameter $\lambda > 0$.
Then, the function $\phi = \phi(y,s)$ satisfies
\begin{align*}%
	\lambda^{2(1+\alpha)/(2+\alpha)} \phi_{ss} - \Delta_{y} \phi + |y|^{\alpha} \phi_s =0.
\end{align*}%
Thus, letting $\lambda \to 0$, we have the heat equation
\begin{align*}%
	|y|^{\alpha} \phi_s - \Delta_y \phi =0.
\end{align*}%
We note that
$\lambda \to 0$ with fixing $s$ is corresponding to $t \to \infty$.

The above observation suggests that
the solution of \eqref{dw} has the so-called {\em diffusion phenomena},
that is,
the solution of the damped wave equation \eqref{dw} is approximated by
a solution of the heat equation \eqref{h}
as time tends to infinity.
Moreover, in the constant damping case $\alpha =0$,
there is a derivation of the damped wave equation
from the heat balance law and the time-delayed Fourier's law
(Cattaneo-Vernotte law).
This derivation also indicates the diffusion phenomena.
For the detail, see for example, \cite{Ca58, Li97, Ve58}.

%%%%%%%%%%%% Previous results %%%%%%%%%%%% %PR
Indeed, for the damped wave equation with constant damping
in the whole space
\begin{align}%
\label{cdw}
	u_{tt} - \Delta u + u_t = 0,\quad x \in \mathbb{R}^N, t>0,
\end{align}%
many mathematicians studied the asymptotic behavior of solutions
and verified the diffusion phenomena.
We refer the reader to
\cite{Ma76, HsLi92, Ni96, Ni97, YaMi00, Kar00, Ni03MathZ, MaNi03, HoOg04, Na04, SaWaMZ}.
For an exterior domain $\Omega \subset \mathbb{R}^N$,
namely, in the case $a(x) \equiv 1$ in our problem \eqref{dw},
the diffusion phenomena was proved by
\cite{Ik02, IkNi03, ChHa03, RaToYo11}.

Asymptotic behavior of solutions to
the wave equation
with variable coefficient damping
\begin{align*}%
	\left\{ \begin{array}{ll}
		u_{tt} - \Delta u + b(x,t) u_t = 0,&x\in \Omega, t>0,\\
		u(x,t) =0, &x\in \partial\Omega,\\
		u(x,0) = u_0(x), \ u_t(x,0) = u_1(x),& x\in \Omega
	\end{array}\right.
\end{align*}%
has been also studied for a long time.
Concerning the behavior of the total energy
\begin{align*}%
	E(t;u) &= \frac12 \int_{\Omega} (u_t^2 + |\nabla u|^2 ) dx,
\end{align*}%
Mochizuki \cite{Mo76} considered the case $\Omega = \R^N$
with $N \neq 2$
and proved that if
\begin{align*}%
	0\le b(x,t) \le C (1+|x|)^{-1-\delta}
\end{align*}%
with some $\delta > 0$
and $|b_t(x,t)| \le C$,
then the M$\mathrm{\phi}$ller wave operator exists and
it is not identically zero.
Namely, there exists initial data
$(u_0, u_1)$ with finite energy and
for the associated solution $u$, the total energy $E(t;u)$ does not decay to zero
as $t \to \infty$,
and there also exists a solution $w$ of the free wave equation
\begin{align*}%
	w_{tt} - \Delta w = 0
\end{align*}%
such that $\lim_{t\to \infty} E(t; u-w) = 0$ holds
(see \cite{RaTa76} for the case $b = b(x) \in C_0^{\infty}(\mathbb{R}^N)$
and \cite{MoNa96, Mat02, Nak06, KaNaSo04, Nis09} for
further improvements).

On the other hand,
Matsumura \cite{Ma77} and Uesaka \cite{Ue79} studied the case
$b(x,t) \ge C(1+t)^{-1}$
and proved that the total energy decays to zero.
Mochizuki and Nakazawa \cite{MoNa96}
and Ueda \cite{Ueda16} gave a logarithmic improvement
of the assumption on the damping.
Also, Nakao \cite{Nakao01MathZ}, Ikehata \cite{Ik03JDE}
and Mochizuki and Nakao \cite{MoNak07}
proved the total energy decay for
the damping localized near infinity.

As for sharp decay estimates,
for the problem \eqref{dw} with $\Omega = \mathbb{R}^N$,
Todorova and Yordanov \cite{ToYo09} proved that
if the damping $a= a(x)$ is radially symmetric and
satisfies \eqref{a} with $\alpha \in (-1,0]$,
then the solution is estimated as
\begin{align*}%
	E(t;u) &= O(t^{-\frac{N+\alpha}{2+\alpha} -1 +\delta}),\\
	\| \sqrt{a} u(t) \|_{L^2(\Omega)}
	&= O(t^{-\frac{N+\alpha}{2(2+\alpha)} + \delta})
\end{align*}%
as $t\to \infty$.
Here $\delta$ is an arbitrary small number
(see also \cite{RaToYo09, RaToYo10} for higher order energy estimates
and see \cite{IkToYo13} for the case $\alpha = -1$).
For the asymptotic profile of solutions,
the second author \cite{Wa14} proved
that if
$\Omega = \mathbb{R}^N$ and $a(x) = (1+|x|^2)^{\alpha/2}$
with $\alpha \in (-1,0]$,
then the asymptotic profile of the solution is given by
that of the corresponding parabolic problem.
After that, in our previous results
\cite{SoWa16, SoWa17}, we extended the result of \cite{Wa14}
to exterior domains and more general (but bounded) damping
satisfying the condition \eqref{a} with $\alpha \in (-1,0]$.
Radu, Todorova and Yordanov \cite{RaToYo16} and Nishiyama \cite{Nis16}
proved the diffusion phenomena in an abstract setting.

However, for the damping term increasing at the spatial infinity,
there is no result about the diffusion phenomena,
while Khader \cite{Kh11, Kh13} studied
energy estimates and global existence of small solutions
for some nonlinear problems.

%%%%%%%% Main results %%%%%%%%%%%%%%%%%%%%

In this paper,
we establish sharp semigroup estimates for
the heat equation \eqref{h}
and prove the almost sharp weighted energy estimates
for the damped wave equation \eqref{dw}.
As their application, we show that
the solution of \eqref{dw} actually has the diffusion phenomena
in a suitable weighted $L^2$ space.

Our main result reads as follows:
%%%%%%%%%%%%%%%%%%%%%%%%%%%%%%%%%%%%%%%%
\begin{theorem}\label{thm1}
Let
$u$
be the solution to \eqref{dw} and
let
$v$
be the solution to \eqref{h} with
$v_0(x) = u_0(x) + a(x)^{-1} u_1(x)$.
Then, for any $\delta>0$,
there exists $C = C(N, \alpha, a_0, R_0, \delta)>0$ such that we have for $t\ge 1$
\begin{align}%
\label{dp}
	\left\| \sqrt{a(\cdot)} (u(\cdot,t)-v(\cdot,t)) \right\|_{L^2}
	\le C (1+t)^{-\frac{N+\alpha}{2(2+\alpha)} -\frac{1+\alpha}{2+\alpha} + \delta}
		\left\| (u_0, u_1) \right\|_{H^2 \times H^1}.
\end{align}%
\end{theorem}
%%%%%%%%%%%%%%%%%%%%%%%%%%%%%%%%%%%%%%%%

As a byproduct of Theorem \ref{thm1},
we obtain the almost optimal $L^2$-estimate for the solution.
%%%%%%%%%%%%%%%%%%%%%%%%%%%%%%%%%%%%%%%%%%
\begin{corollary}\label{cor2}
The solution $u$ of \eqref{dw} satisfies
\begin{align*}%
	\left\| \sqrt{a(\cdot)} u(\cdot,t) \right\|_{L^2}
	\le C (1+t)^{-\frac{N+\alpha}{2(2+\alpha)}+\delta}
		\left\| (u_0, u_1) \right\|_{H^2 \times H^1},
\end{align*}%
where
$\delta > 0$ is an arbitrary small number
and
$C = C(N, \alpha, a_0, R_0, \delta)>0$.
When $N=2$, we may take $\delta = 0$.
\end{corollary}
%%%%%%%%%%%%%%%%%%%%%%%%%%%%%%%%%%%%%%%%%%

Our strategy for the proof of Theorem \ref{thm1} is the following.
First, we treat the term $u_{tt}$ as a perturbation and write the equation \eqref{dw} as
\begin{align*}%
	u_t - a(x)^{-1} \Delta u = - a(x)^{-1}u_{tt}.
\end{align*}%
This is natural because we expect the diffusion phenomena,
and for the solution $v$ to the parabolic problem \eqref{h},
the term $v_{tt}$ decays faster than $v_t$ and $a(x)^{-1} \Delta v$.
Then, by Duhamel principle, the above equation
can be formally transformed into
the integral equation
\begin{align*}%
	u(t) = e^{t a(x)^{-1}\Delta} u_0 - \int_0^t e^{(t-s) a(x)^{-1}\Delta} [ a(x)^{-1} u_{ss}(s) ]\,ds.
\end{align*}%
We further apply the integration by parts to obtain
\begin{align*}%
	u(t) &= e^{t a(x)^{-1}\Delta} ( u_0 + a(x)^{-1}u_1) \\
		&\quad - \int_{t/2}^t e^{(t-s) a(x)^{-1}\Delta} [ a(x)^{-1} u_{tt}(s) ]\,ds \\
		&\quad - e^{\frac{t}{2} a(x)^{-1}\Delta} [ a(x)^{-1} u_{t}(t/2) ] \\
		&\quad - \int_0^{t/2}
				a(x)^{-1} \Delta e^{(t-s) a(x)^{-1}\Delta} [ a(x)^{-1} u_{t}(s) ]\,ds
\end{align*}%
By putting $v(t) = e^{t a(x)^{-1}\Delta} ( u_0 + a(x)^{-1}u_1)$,
it suffices to estimate the each term in the right-hand side.

To this end, in Section 2, we first investigate the heat semigroup
$e^{t a(x)^{-1}\Delta}$.
We let it make sense in an appropriate weighted space
via a suitable changing variable.
Moreover, by the Beurling-Deny criterion and the Gagliardo-Nirenberg inequality,
we derive the following optimal estimate for the semigroup
$e^{t a(x)^{-1}\Delta}$,
which is crucial for our argument:
\begin{align*}%
	\| e^{t a(x)^{-1}\Delta}v_0 \|_{L^2(\Omega, d\mu)} \le C
		t^{-(N/2)(1/q - 1/2)}
			\left( \int_{\Omega} |v_0(x)|^q |x|^{(N-2)\alpha(2-q)/4} \, d\mu \right)^{1/q},
\end{align*}%
where
$d\mu = a(x)dx$
and
$q \in [1,2] \ (N=2), \ q\in ( p_{\alpha}', 2]\ (N\ge 3)$
with
$p_{\alpha} = 2N(2+\alpha)/(\alpha(N-2))$.
In contrast with the decaying damping cases studied by \cite{SoWa16, SoWa17}
in which $q$ can be taken freely in $[1,2]$,
we cannot have $L^2$-$L^1$ estimate for $N\ge 3$.

After that, to estimate the terms including $u$ itself,
we apply energy estimates for
the damped wave equation \eqref{dw}
with a Ikehata-Todorova-Yordanov type weight function
$\Phi = \exp ( \beta A(x)/(1+t) )$,
where
$\beta$ is a positive constant and
$A(x)$ is a solution of the elliptic equation
\begin{align}%
\label{poi}
	\Delta A(x) = a(x).
\end{align}%
Roughly speaking, the weight function
$\Phi$ comes from the dual problem of the heat equation
\begin{align*}%
	a(x) v_t + \Delta v = 0.
\end{align*}%
We remark that Lions and Masmoudi \cite{LiMa98, LiMa01}
firstly introduced this type of weight functions
to study the uniqueness for the Navier-Stokes equations.

In Section 3, we discuss the construction of the auxiliary function $A(x)$.
When $a(x)$ is decaying at spatial infinity and radially symmetric,
Todorova and Yordanov \cite{ToYo09} solved the equation \eqref{poi}
by reducing the problem to an ordinary differential equation,
and applied it to obtain energy estimates for damped wave equations.
However, in our problem, we assume $a(x)$ is unbounded.
In this case we cannot directly construct a solution by the Newton potential.
Moreover, $a(x)$ is not radially symmetric as in \cite{ToYo09}
and we cannot reduce the problem to an ordinary differential equation.
Indeed, in \cite[Remark 3.1]{SoWa16},
we show a example of non-radial $a(x)$
satisfies $a(x) \to 1 \ (|x| \to \infty)$
but the corresponding $A(x)$ has bad behavior at $|x| \to \infty$.
To overcome these difficulties,
we follow our previous results \cite{SoWa17} and
weaken the problem \eqref{poi} to the inequality
\begin{align*}%
	(1-\varepsilon) a(x) \le \Delta A_{\varepsilon}(x) \le (1+\varepsilon) a(x)
\end{align*}%
with arbitrary small $\varepsilon >0$,
namely, we shall find a function $A_{\varepsilon}(x)$ satisfying the above
inequality and having good behavior as $|x| \to \infty$.

In Section 4, using this auxiliary function $A_{\varepsilon}(x)$,
we apply weighted energy methods developed by
\cite{ToYo01, Ik05IJPAM, IkTa05, RaToYo09, Ni10, Wa14, SoWa16, SoWa17}
and prove almost sharp higher order energy estimates of solutions.

Finally, in Section 5,
combining the heat semigroup estimates and the weighted energy estimates
for solutions to the damped wave equation \eqref{dw},
we have the desired estimate and finish the proof of Theorem \ref{thm1}.

%NNN
We end up this section with introducing the notations and the terminologies
used throughout this paper.
We shall denote by
$C$ various constants,
which may change from line to line.
In particular,
we sometimes write
$C = C(\ast, \ldots, \ast)$
to emphasize the dependence on the parameters appearing in parentheses.

For $x_0 \in \mathbb{R}^N$ and $R>0$, we denote by
$B(x_0,R)$
the open ball centered at $x_0$ with the radius $R$.
Also, $\bar{B}(x_0, R)$ stands for the closure of $B(x_0,R)$.

We denote by
$L^p(\Omega) \, (1\le p \le \infty)$, $H^k(\Omega)\, (k\in \mathbb{Z}_{\ge 0})$
the usual Lebesgue space and Sobolev space with the norms
\begin{align*}%
	\| f \|_{L^p} &=
	\begin{cases}
		\displaystyle
		\left( \int_{\Omega} |f(x)|^p\,dx \right)^{1/p} &(1\le p <\infty),\\
		\displaystyle
		{\rm ess\,sup\,}_{x\in \Omega} |f(x)| &(p=\infty),
	\end{cases}\\
	\| f \|_{H^k} &=
		\left( \sum_{|\gamma|\le k}\| \partial_x^{\gamma} f \|_{L^2}^2 \right)^{1/2},
\end{align*}%
respectively.
For an interval $I$ and a Banach space $X$,
we define
$C^{r}(I;X)$
as the space of $r$-times continuously differentiable mapping from
$I$ to $X$ with respect to the topology in $X$.

%%%%%%%%%%%%%%%%%%%%%%%%%%%%%%%%%%%%%%%%%
%%%%%%%%%%%%%%%%%%%%%%%%%%%%%%%%%%%%%%%%%
%%%%%%%%%%%%%%%%%%%%%%%%%%%%%%%%%%%%%%%%%
%%%%%%%%%%%%%%%%%%%%%%%%%%%%%%%%%%%%%%%%%
\section{$L^p$-$L^q$ estimates for the parabolic problem}
In this section,
we consider the asymptotic behavior of solutions to
the parabolic problem \eqref{h}.
We note that the assumption \eqref{a} on $a(x)$ implies
that there exist constants
$0<a_1\le a_2$
such that
\begin{align}%
\label{a_est}
	a_1|x|^{\alpha} \le a(x) \le a_2|x|^{\alpha}
\end{align}%
holds for $x \in \bar{\Omega}$.
We remark that the results of this section
requires only \eqref{a_est} and
we do not use the property \eqref{a}.
For $1\le p < \infty$, we define the weighted Lebesgue spaces
\begin{align*}%
	L^p_{d\mu}(\Omega)
	:=
	\left\{ f \in L^p_{{\rm loc}}(\Omega) \,;\,
		\| f \|_{L^p_{d\mu}} = \left( \int_{\Omega} |f(x)|^p a(x) \,dx \right)^{1/p} < \infty
	\right\}.
\end{align*}%
We deal with the operator
\begin{align*}%
	L = - a(x)^{-1} \Delta
\end{align*}%
in $L^2_{d\mu}(\Omega)$.
The operator $L$ is regarded as
an associated operator of the symmetric sesquilinear form
\begin{align*}%
	\mathfrak{a}(u,v) := \int_{\Omega} \nabla u \cdot \nabla\bar{v} \,dx
\end{align*}%
with the domain
\begin{align*}%
	\mathcal{D} := D(\mathfrak{a})
	= \{ u\in C^{\infty}_c(\bar{\Omega}) ;\,
		u|_{\partial\Omega} = 0 \}.
\end{align*}%
The basic property of $L$ is given in the following.
%%%%%%%%%%%%%%%%%%%%%%%%%%%%%%%%%%%%%%%%%%%%%%
\begin{lemma}\label{lem_A}
The sesquilinear form
$\mathfrak{a}$
is densely defined, continuous, accretive, symmetric and closable in
$L^2_{d\mu}(\Omega)$.
Therefore, there exists a realization
$L_{\ast}$ of $L$ such that $L_{\ast}$
is nonnegative and self-adjoint.
Moreover, $\mathcal{D}$ is a core for $L_{\ast}$.
\end{lemma}
%%%%%%%%%%%%%%%%%%%%%%%%%%%%%%%%%%%%%%%%%%%%%%
While this lemma was already discussed in \cite{SoWa16},
we give a sketch of the proof for the reader's convenience.
%%%%%%%%%%%%%%%%%%%%%%%%%%%%%%%%%%%%%%%%%%%%%%
\begin{proof}[Proof of Lemma \ref{lem_A}]
It is obvious that $\mathfrak{a}$
is densely defined, continuous, accretive, symmetric
in $L^2_{d\mu}(\Omega)$.
Moreover, by \cite[Proposition 1.3]{Ou},
we easily see that $\mathfrak{a}$ is closable in $L^2_{d\mu}(\Omega)$.
Let
$\mathfrak{a}_{\ast}$
be the closure of
$\mathfrak{a}$.
Then, \cite[Lemma 2.1]{SoWa16} shows that
the bilinear form $\mathfrak{a}_{\ast}$ is characterized as
\begin{align}%
\label{da}
	D(\mathfrak{a}_{\ast})
	&= \left\{ u \in L^2_{d\mu}(\Omega) \cap \dot{H}^1(\Omega) ;
	\int_{\Omega} \frac{\partial u}{\partial x_j} \varphi \,dx
	= - \int_{\Omega} u \frac{\partial \varphi}{\partial x_j}\,dx \right.\\
\nonumber
	&\left. \hspace{110pt}\ \mbox{for}\ \varphi \in C_c^{\infty}(\mathbb{R}^N),
	\ j=1,\ldots,N
	\right\},\\
\nonumber
	\mathfrak{a}_{\ast}(u,v)
	&= \int_{\Omega} \nabla u \cdot \nabla v\,dx.
\end{align}%
We define the Friedrichs extension
$L_{\ast}$
of $L$ by
\begin{align*}%
	D(L_{\ast}) &:= \left\{
		u \in D(\mathfrak{a}_{\ast}) ; 
		\ \mbox{there exists}\ f\in L^2_{d\mu}(\Omega) \ \mbox{such that}\right.
		\\
	&\left. \hspace{70pt}\ 
	\mathfrak{a}_{\ast} (u,v) = (f,v)_{L^2_{d\mu}}
	\ \mbox{for any}\ v\in D(\mathfrak{a}_{\ast}) \right\},\\
	L_{\ast} u &= f.
\end{align*}%
Then, by \cite[Theorem X.23]{ReSi},
we see that
$L_{\ast}$
is a nonnegative self-adjoint operator on
$L^2_{d\mu}(\Omega)$.
Moreover, we have for $u, v \in \mathcal{D}$,
\begin{align*}%
	(-Lu, v)_{L^2_{d\mu}}
	= \int_{\Omega} (-L u) v d\mu
	= \int_{\Omega} (\Delta u ) v dx
	= \mathfrak{a}_{\ast}(u,v).
\end{align*}%
By a density argument, we can take $v \in D(\mathfrak{a}_{\ast})$
in the above equalities.
This implies $u \in D(L_{\ast})$ and $Lu = L_{\ast}u$,
and hence,
$L_{\ast}$ is an extension of $L$.
\end{proof}
%%%%%%%%%%%%%%%%%%%%%%%%%%%%%%%%%%%%%%%%%%%%%%
\subsection{Transformation into a usual Lebesgue spaces}
We introduce a diffeomorphism
$\Psi \in C^{\infty}(\Omega ; \Omega_{\Psi})$ 
as
\begin{align*}%
	\Psi(x) := |x|^{\alpha/2}x,\quad
	\Omega_{\Psi} :=
		\left\{ y \in \mathbb{R}^N ; |y|^{-\alpha/(2+\alpha)}y \in \Omega \right\}.
\end{align*}%
Since we assume that
$\partial \Omega$
is smooth, so is $\partial \Omega_{\Psi}$.

The following lemma is a fundamental fact on changes of variables.
%%%%%%%%%%%%%%%%%%%%%%%%%%%%%%%%%%%%%%%%%%%%%%%%%
\begin{lemma}\label{lem_dpsi}
\begin{itemize}
\item[(i)]
One has
\begin{align*}%
	\frac{\partial \Psi}{\partial x}(x)
	&= |x|^{\alpha/2} \left( I + \frac{\alpha}{2} Q(x) \right),
	\quad Q(x) = \left( \frac{x_jx_k}{|x|^2} \right)_{j,k=1,\ldots,N},\\
	\det \left( \frac{\partial \Psi}{\partial x}(x) \right)
	&= \frac{2+\alpha}{2} |x|^{N \alpha/2}.
\end{align*}%
Here we denote by $I$ the identity matrix of the order $N$.
\item[(ii)]
Define $m(x) := |x|^{-\alpha}a(x)$
for $x \in \Omega$,
$\tilde{m}(y) := m( \Psi^{-1}(y) )$
for $y \in \Omega_{\Psi}$
and $d \nu = \tilde{m}(y)dy$.
Then, the norm of $L^2(\Omega_{\Psi}, d\nu)$ is equivalent to
$L^2$-norm with usual Lebesgue measure on $\Omega_{\Psi}$.
Moreover,
\begin{align*}%
	\int_{\Omega_{\Psi}} |v(y)|^2 \,d\nu
	= \frac{2+\alpha}{2}
		\int_{\Omega} |v \left( \Psi(x) \right) |^2 |x|^{(N-2)\alpha/2}\, d\mu
\end{align*}%
holds for $v \in C_c^{\infty} ( \Omega_{\Psi} )$.
\end{itemize}
\end{lemma}
%%%%%%%%%%%%%%%%%%%%%%%%%%%%%%%%%%%%%%%%%%%%%%%%
%%%%%%%%%%%%%%%%%%%%%%%%%%%%%%%%%%%%%%%%%%%%%%%
\begin{proof}
(i) is straightforward.
For (ii), we note that the assumption \eqref{a_est} implies
\begin{align*}%
	a_1 \le m(x) \le a_2
\end{align*}%
for any $x\in \bar{\Omega}$.
From this, we obtain
$a_1 \le \tilde{m}(y) \le a_2$
for any $y \in \overline{\Omega_{\Psi}}$,
which leads to the equivalence between
$L^2(\Omega_{\Psi}, d\nu)$ and $L^2(\Omega_{\Psi})$.
Moreover, we have
\begin{align*}%
	\int_{\Omega_{\Psi}} |v(y)|^2\, d\nu
	&= \int_{\Omega_{\Psi}} |v(y)|^2 m \left( \Psi^{-1}(y) \right)\,dy\\
	&= \int_{\Omega} |v \left( \Psi(x) \right)|^2 m(x)
		\left| \det \left( \frac{\partial \Psi}{\partial x} \right) \right|\,dx \\
	&= \frac{2+\alpha}{2} \int_{\Omega} | v\left( \Psi(x) \right) |^2
			|x|^{(N-2)\alpha/2}\, d\mu,
\end{align*}%
which completes the proof.
\end{proof}
%%%%%%%%%%%%%%%%%%%%%%%%%%%%%%%%%%%%%%%%%%%%%%%

In view of the above lemma, we introduce an isometry from
$L^2(\Omega_{\Psi}, d\nu)$ to $L^2(\Omega, d\mu)$ as
\begin{align}%
\label{j}
	Jv(x) := \sqrt{ \frac{2+\alpha}{2} } |x|^{(N-2)\alpha/4}
		v \left( \Psi(x) \right)
\end{align}%
for $v \in L^2(\Omega_{\Psi}, d\nu)$.
It is easy to see that the inverse of $J$ is given by
\begin{align}%
\label{ji}
	J^{-1} u (y) =
		\sqrt{\frac{2}{2+\alpha}} |y|^{-(N-2)\alpha/(2(2+\alpha))}
		u \left( |y|^{-\alpha/(2+\alpha)} y \right)
\end{align}%
for $u \in L^2(\Omega, d\mu)$.
Then, the operator $L$ is transformed into the following uniformly elliptic operator
on $L^2(\Omega_{\Psi}, d\nu)$.

%%%%%%%%%%%%%%%%%%%%%%%%%%%%%%%%%%%%%%%%%
\begin{lemma}\label{lem_jiaj}
The operator
$B=J^{-1}LJ$ with the domain
$\mathcal{D}_{\Psi} := \{ u \circ \Psi^{-1} \,;\, u \in \mathcal{D} \}$
is given by
\begin{align}%
\label{B}
	B v(y) = \tilde{m}(y)^{-1}
		\left( - {\rm div\,}( b(y) \nabla v(y))
		- \frac{(N-2)^2 \alpha (4+\alpha)}{16|y|^2} v(y) \right)
\end{align}%
for $y \in \Omega_{\Psi}$ and $v \in \mathcal{D}_{\Psi}$,
where
\begin{align*}%
	b(y) = I + \alpha \left( 1+ \frac{\alpha}{4} \right) Q(\Psi^{-1}(y)).
\end{align*}%
Moreover, $B$ is a uniformly elliptic operator in $\Omega_{\Psi}$
with bounded coefficients.
\end{lemma}
%%%%%%%%%%%%%%%%%%%%%%%%%%%%%%%%%%%%%%%%%
%%%%%%%%%%%%%%%%%%%%%%%%%%%%%%%%%%%%%%%%%
\begin{proof}
Let $\varphi \in \mathcal{D}_{\Psi}$ be arbitrary fixed
and let $v \in \mathcal{D}_{\Psi}$.
Then, we obtain
\begin{align*}%
	\int_{\Omega_{\Psi}}
		\left( Bv(y) \right) \varphi(y)\,d\nu
	&= \frac{2+\alpha}{2} \int_{\Omega}
		\left( J^{-1} L J v(\Psi(x)) \right) \varphi(\Psi(x)) |x|^{(N-2)\alpha/2}\,d\mu \\
	&= \int_{\Omega} \left( AJv(x) \right) J \varphi(x)\,d\mu \\
	&= \int_{\Omega} \nabla Jv(x) \cdot \nabla J\varphi(x)\,dx.
\end{align*}%
Noting that
\begin{align*}%
	\sqrt{\frac{2}{2+\alpha}} \nabla J \varphi (x)
	&= \frac{(N-2)\alpha}{4} |x|^{(N-2)\alpha/4-2} x \varphi(\Psi(x)) \\
	&\quad + |x|^{N\alpha/4} \left( I + \frac{\alpha}{2}Q(x) \right)\nabla \varphi ( \Psi(x) )
\end{align*}%
and putting
$\beta = (N-2)\alpha/4$, we have
\begin{align*}%
	&\int_{\Omega} \nabla Jv(x) \cdot \nabla J \varphi(x)\,dx\\
	&= \int_{\Omega} \left( \beta \frac{y}{|y|^2}v(y)
		+ \left( I + \frac{\alpha}{2}Q(\Psi^{-1}(y)) \right)\nabla v(y) \right) \\
	&\qquad\quad \cdot \left( \beta \frac{y}{|y|^2} \varphi (y)
		+ \left( I + \frac{\alpha}{2}Q(\Psi^{-1}(y)) \right)\nabla \varphi (y) \right)\,dy \\
	&= \int_{\Omega}
		\left( I + \frac{\alpha}{2}Q(\Psi^{-1}(y)) \right)^2 \nabla v(y) \cdot \nabla \varphi(y)\,dy
		+ \beta \left( 1 + \frac{\alpha}{2} \right)
		\int_{\Omega} \frac{y}{|y|^2} \cdot \nabla ( v(y) \varphi(y) )\,dy\\
	&\quad + \beta^2 \int_{\Omega} \frac{1}{|y|^2} v(y) \varphi(y)\,dy \\
	&= - \int_{\Omega} {\rm div\,} (b(y)\nabla v(y) ) \varphi(y)\,dy
		+ \left( \beta^2 - (N-2)\beta \left( 1+\frac{\alpha}{2} \right) \right)
		\int_{\Omega} \frac{1}{|y|^2} v(y) \varphi(y)\,dy.
\end{align*}%
Finally, the uniform ellipticity of $B$ follows from
$Q \ge 0$ as a symmetric matrix.
The boundedness of the coefficients immediately shown
by the assumption $0 \notin \bar{\Omega}$ and the definition of $Q$.
\end{proof}
%%%%%%%%%%%%%%%%%%%%%%%%%%%%%%%%%%%%%%%%%

From the above lemma,
by the same procedure as \cite[pp.100--101]{Ou},
we can associate an operator
$B_2$ on $L^2(\Omega_{\Psi}, d\nu)$
with the sesquilinear form
\begin{align*}%
	\mathfrak{b}(u,v)
	&= \int_{\Omega_{\Psi}} \left(
		\left( b(y) \nabla u (y) \right) \cdot \nabla v(y)
		- \frac{ (N-2)^2 \alpha (4+\alpha)}{16|y|^2} u(y) v(y) \right)\, dy,\\
	D(\mathfrak{b})
	&= \{ u \circ \Psi^{-1} ; u \in D(\mathfrak{a}_{\ast}) \},
\end{align*}%
and $-B_2$ is the generator of an analytic semigroup
$T_2(t) = e^{-t B_2}$ on $L^2(\Omega, d\nu)$.
Therefore, it follows from \cite[Theorem 1.52]{Ou} that
$-B_2$ generates an analytic semigroup on $L^2(\Omega_{\Psi}, d\nu)$.
%\label{Fattorini_bookの"Cauchy problem" p.235付近 Section 4.8}

Furthermore,
applying \cite[Theorem 4.28]{Ou},
we extend it to an analytic semigroup
generated by $-B$
in $L^p(\Omega_{\Psi}, d\nu)$.
Summarizing the above, we have the following:
%%%%%%%%%%%%%%%%%%%%%%%%%%%%%%%%%%%%%%%%%
\begin{lemma}\label{lem_sg}
For every
$1<p<\infty$,
the operator
$-B_p$
defined by \eqref{B}
in
$L^p(\Omega_{\Psi}, d\nu)$
endowed with the domain
$D(B_p) = W^{2,p}(\Omega_{\Psi}, d\nu) \cap W_0^{1,p}(\Omega_{\Psi}, d\nu)$
generates an analytic semigroup
$( T_p(t) )_{t \ge 0}$
and
$\mathcal{D}_{\Psi}$
is a core for
$B_p$.
Moreover, 
$T_p(t)$
is consistent for all
$1<p<\infty$.
\end{lemma}
%%%%%%%%%%%%%%%%%%%%%%%%%%%%%%%%%%%%%%%%%%
%%%%%%%%%%%%%%%%%%%%%%%%%%%%%%%%%%%%%%%%%%
%%%%%%%%%%%%%%%%%%%%%%%%%%%%%%%%%%%%%%%%%%

By virtue of Lemma \ref{lem_sg},
we introduce $(T(t))_{t\ge 0}$ as a semigroup generated by $-B$ with
Dirichlet boundary condition.

Next, we consider
$L^2$-$L^q$ estimates for the semigroup $(T(t))_{t\ge 0}$.
As we said in the introduction,
this is crucial to prove the diffusion phenomena
for the damped wave equation \eqref{dw}.

%%%%%%%%%%%%%%%%%%%%%%%%%%%%%%%%%%%%%%%%%
\begin{proposition}\label{prop_lplq}
Let
$N\ge 2$
and
$\alpha > 0$.
Then we have the following
$L^{p}$-$L^2$ estimates:
\begin{itemize}
\item[(i)]
If $N=2$, then
the the semigroup
$(T(t))_{t\ge 0}$ is sub-Markovian and
satisfies the estimate
\begin{align*}%
	\| T(t) f \|_{L^{\infty}(\Omega_{\Psi}, d\nu)}
	\le C t^{-1/2} \| f \|_{L^2(\Omega_{\Psi}, d\nu)}; 
\end{align*}%
\item[(ii)]
If $N\ge 3$, then for every $2 \le p < p_{\alpha} := 2N(2+\alpha)/(\alpha(N-2))$,
then the semigroup
$(T(t))_{t\ge 0}$
satisfies the estimate
\begin{align*}%
	\| T(t) f \|_{L^p(\Omega_{\Psi}, d\nu)}
	\le C t^{-(N/2)(1/2 - 1/p)} \| f \|_{L^2(\Omega_{\Psi}, d\nu)}.
\end{align*}%
\end{itemize}
\end{proposition}
%%%%%%%%%%%%%%%%%%%%%%%%%%%%%%%%%%%%%%%%%
%%%%%%%%%%%%%%%%%%%%%%%%%%%%%%%%%%%%%%%%%
\begin{remark}\label{rem_lplq}
Let $N\ge 3, \Omega = \mathbb{R}^N\setminus B(0,1)$
and $\varepsilon > 0$.
Then, the function
\begin{align*}%
	\psi_0(y) = |y|^{-(N-2)\alpha/4} \left( 1 - |y|^{-N+2} \right)
\end{align*}%
belongs to $L^{p_{\alpha}+\varepsilon}(\Omega_{\Psi}, d\nu)$
and satisfies $B\psi_0 = 0$ in $\Omega_{\Psi}$.
This means that
$\psi_0$
is a stationary solution of the equation
\begin{align*}%
	\left\{ \begin{array}{ll}
		\psi_t + B_{p_{\alpha}+\varepsilon} \psi = 0,
			&y \in \Omega_{\Psi},\ t > 0,\\
		\psi (y,t) = 0,
			&y \in \partial \Omega_{\Psi},\ t>0,\\
		\psi (y,0) = \psi_0 (y),
			&y \in \Omega_{\Psi}.
		\end{array}\right.
\end{align*}%
By virtue of
$\psi_0$,
any decay estimate for the semigroup $(T(t))_{t\ge 0}$
in
$L^{p_{\alpha}+\varepsilon}(\Omega_{\Psi}, d\nu)$
cannot be expected.
We would expect that for $N \ge 3$
the $L^2$-$L^{p_{\alpha}}$ estimate is given by
\begin{align*}%
	\| T(t) f \|_{L^{p_{\alpha}}(\Omega_{\Psi}, d\nu)}
	\le C t^{-(N/2)(1/2-1/p_{\alpha})}(1+\log (1+t))^{1/p_{\alpha}}
		\| f \|_{L^2(\Omega_{\Psi}, d\nu)}.
\end{align*}%
\end{remark}
%%%%%%%%%%%%%%%%%%%%%%%%%%%%%%%%%%%%%%%%%

We postpone the proof of Proposition \ref{prop_lplq}
in the following subsections, and here
we give a practical version of it.
For the original variable $x$, 
Proposition \ref{prop_lplq} and the identity
\begin{align*}%
	\left( \frac{2+\alpha}{2} \right)^{1-p/2}
	\int_{\Omega} | Jv(x) |^p |x|^{-(N-2)\alpha(p-2)/4} a(x)\, dx
	= \int_{\Omega_{\Psi}} |v(y)|^p\, d\nu
\end{align*}%
imply the following estimate.
%%%%%%%%%%%%%%%%%%%%%%%%%%%%%%%%%%%%%%%%
\begin{proposition}\label{prop_lplq2}
Let $N \ge 2$ and $\alpha>0$ and let $v$ be a solution to \eqref{h}.
Then, we have the following.
\begin{itemize}
\item[(i)]
If $N=2$, then we have
\begin{align*}%
	\| v(t) \|_{L^2(\Omega, d\mu)} \le C t^{-1/2} \| v_0 \|_{L^1(\Omega, d\mu)};
\end{align*}%
\item[(ii)]
If $N\ge 3$, then for every $(p_{\alpha})^{\prime} < q \le 2$, we have
\begin{align*}%
	\| v(t) \|_{L^2(\Omega, d\mu)} \le C
		t^{-(N/2)(1/q - 1/2)}
			\left( \int_{\Omega} |v_0(x)|^q |x|^{(N-2)\alpha(2-q)/4} \, d\mu \right)^{1/q}.
\end{align*}%
\end{itemize}
\end{proposition}
%%%%%%%%%%%%%%%%%%%%%%%%%%%%%%%%%%%%%%%%

%%%%%%%%%%%%%%%%%%%%%%%%%%%%%%%%%%%%%%%%
%%%%%%%%%%%%%%%%%%%%%%%%%%%%%%%%%%%%%%%%
\subsection{Proof of Proposition \ref{prop_lplq} for $N=2$}
In this subsection we use the following
Gagliardo--Nirenberg inequality with $N=2$.
%%%%%%%%%%%%%%%%%%%%%%%%%%%%%%%%%%%%%%%%
\begin{lemma}[Gagliardo--Nirenberg inequality]\label{lem_gn}
Let $N=2$.
For every $2 \le q < \infty$, there exists a constant $C_q$ such that
\begin{align*}%
	\| w \|_{L^q(\Omega)} \le C_q
		\| \nabla w \|_{L^2(\Omega)}^{1-2/q} \| w \|_{L^2(\Omega)}^{2/q}
\end{align*}%
holds for any $w \in W^{1,2}_0(\Omega)$.
\end{lemma}
%%%%%%%%%%%%%%%%%%%%%%%%%%%%%%%%%%%%%%%%
For the proof, see for example
\cite{Frbook}.
From this, we obtain the following estimate.
%%%%%%%%%%%%%%%%%%%%%%%%%%%%%%%%%%%%%%%%
\begin{lemma}\label{lem_gn2}
For every $2 \le q < \infty$, we have
\begin{align}
\label{gn2}
	\| \varphi \|_{L^q(\Omega_{\Psi}, d\nu)}
		\le C  (B_2 \varphi, \varphi )_{d\nu}^{1/2-1/q}
			\| \varphi \|_{L^2(\Omega_{\Psi}, d\nu)}^{1/q}
\end{align}
for any $\varphi \in L^2(\Omega_{\Psi}, d\nu)$.
\end{lemma}
%%%%%%%%%%%%%%%%%%%%%%%%%%%%%%%%%%%%%%%%
%%%%%%%%%%%%%%%%%%%%%%%%%%%%%%%%%%%%%%%%
\begin{proof}
Since $\mathcal{D}_{\Psi}$ is a core for $B_2$,
it suffices to show \eqref{gn2} for $\varphi \in \mathcal{D}_{\Psi}$.
Let
$\varphi \in \mathcal{D}_{\Psi}$.
Noting that
\begin{align*}%
	b(y) \ge \min \left\{ 1, 1+\alpha + \frac{\alpha^2}{4} \right\} I = I
\end{align*}%
in the sense of symmetric matrix,
we have
\begin{align*}%
	( B_2 \varphi, \varphi )_{d\nu}
	= \int_{\Omega_{\Psi}}  ( b(y) \nabla \varphi) \cdot \nabla \varphi \,dy
	\ge \| \nabla \varphi \|_{L^2(\Omega_{\Psi})}^2.
\end{align*}%
Combining the above estimate with
$a_1 \le \tilde{m} \le a_2$
and Lemma \ref{lem_gn}, we obtain
\begin{align*}%
	\| \varphi \|_{L^q(\Omega_{\Psi}, d\nu)}
	&\le a_2^{1/q} \left( \int_{\Omega_{\Psi}} | \varphi (y) |^q\,dy \right)^{1/q}\\
	&\le a_2^{1/q} C_q
		\left( \int_{\Omega_{\Psi}} | \nabla \varphi(y) |^2\,dy \right)^{1/2-1/q}
		\left( \int_{\Omega_{\Psi}} | \varphi (y) |^2\, dy \right)^{1/q} \\
	&\le a_2^{1/q}a_1^{-1/2} C_q
		(B_2 \varphi, \varphi)_{d\nu}^{1/2-1/q}
		\| \varphi \|_{L^2(\Omega_{\Psi}, d\nu)},
\end{align*}%
which implies \eqref{gn2}.
\end{proof}
%%%%%%%%%%%%%%%%%%%%%%%%%%%%%%%%%%%%%%%%

%%%%%%%%%%%%%%%%%%%%%%%%%%%%%%%%%%%%%%%%
\begin{lemma}\label{lem_sm}
The semigroup $(T(t))_{t\ge 0}$ is sub-Markobian,
that is, $T(t)$ is positively preserving:
\begin{align*}%
	f \in L^2( \Omega_{\Psi}, d\nu),\ f\ge 0
	\ \Rightarrow\ T(t) f \ge 0
\end{align*}%
and $L^{\infty}$-contractive:
\begin{align*}%
	\| T(t) f \|_{L^{\infty}(\Omega_{\Psi}, d\nu)}
	\le \| f \|_{L^{\infty}(\Omega_{\Psi}, d\nu)}
	\quad \mbox{for}\quad f \in L^2(\Omega_{\Psi}, d\nu) \cap L^{\infty}(\Omega_{\Psi}, d\nu).
\end{align*}%
\end{lemma}
%%%%%%%%%%%%%%%%%%%%%%%%%%%%%%%%%%%%%%%%
%%%%%%%%%%%%%%%%%%%%%%%%%%%%%%%%%%%%%%%%
\begin{proof}
Let $v \in D(\mathfrak{b})$.
Using the characterization \eqref{da} of $D(\mathfrak{a}_{\ast})$,
we easily see that
$|v| \in D(\mathfrak{b})$.
Moreover, we have
\begin{align*}%
	\mathfrak{b}(|v|, |v|)
	= \int_{\Omega_{\Psi}}
		\left( b(y) \frac{v}{|v|} \nabla v (y) \right) \cdot \frac{v}{|v|} \nabla v(y) \, dy
	= \mathfrak{b}(v,v).
\end{align*}%
Thus, applying \cite[Theorem 2.7]{Ou}, we have
the positively preserving property of $T(t)$.
Next, we show the $L^{\infty}$-contractive property of $T(t)$.
By using the characterization \eqref{da} of $D(\mathfrak{a}_{\ast})$ again,
we also see that
$v \in D(\mathfrak{b})$
implies
$P v = (1 \wedge |v|) {\rm sign\,}v \in D(\mathfrak{b})$.
Furthermore, we have
\begin{align*}%
	\mathfrak{b}(Pv, v - Pv)
	= \int_{\Omega_{\Psi}}
	\left( b(y) \chi_{|v| < 1} \nabla v (y) \right)
		\cdot \left( (1 - \chi_{|v|<1}) \nabla v(y) \right) \, dy
	= 0.
\end{align*}%
Therefore, applying \cite[Theorem 2.13]{Ou}, we prove
the $L^{\infty}$-contractive property of $T(t)$.
\end{proof}
%%%%%%%%%%%%%%%%%%%%%%%%%%%%%%%%%%%%%%%%

%%%%%%%%%%%%%%%%%%%%%%%%%%%%%%%%%%%%%%%%
\begin{proof}[Proof of Proposition \ref{prop_lplq} for $N =2$ ]
By Lemma \ref{lem_sm}, we see that
the semigroup $T(t)$ is $L^{\infty}$-contractive.
Moreover, by Lemma \ref{lem_gn2},
$T(t)$ satisfies the second condition of \cite[Theorem 6.2]{Ou} with $d=2$.
Thus, applying \cite[Theorem 6.2]{Ou}, we conclude
\begin{align*}%
	\| T(t) f \|_{L^{\infty}(\Omega_{\Psi}, d\nu)}
	\le C t^{-1/2} \| f \|_{L^2(\Omega_{\Psi}, d\nu)},
\end{align*}%
which gives the assertion of Proposition \ref{prop_lplq} for $N=2$.
\end{proof}
%%%%%%%%%%%%%%%%%%%%%%%%%%%%%%%%%%%%%%%%

%%%%%%%%%%%%%%%%%%%%%%%%%%%%%%%%%%%%%%%%
%%%%%%%%%%%%%%%%%%%%%%%%%%%%%%%%%%%%%%%%
\subsection{Proof of Proposition \ref{prop_lplq} for $N \ge 3$}
In this case we need to prepare some result for analytic semigroup 
generated by $-B_p$ via $L^p$-theory.
First we prove that
$- B_p$
generates an analytic contraction semigroup when
\begin{align}%
\label{range.p}
	2-\frac{4}{4+\alpha} < p < 2+\frac{4}{\alpha}.
\end{align}%
%%%%%%%%%%%%%%%%%%%%%%%%%%%%%%%%%%%%%%%%%
\begin{lemma}\label{lem_sect}
Let
$N\ge 3$.
Assume that \eqref{range.p} is satisfied.
Then, there exists a constant
$\ell_{N,p,\alpha} \ge 0$ depending only on $N, p$ and $\alpha$
such that $B_p$ is m-sectorial of type
$S(\ell_{N,p,\alpha})$
in
$L^p(\Omega_{\Psi}, d\nu)$,
that is,
\begin{align}%
\label{sect}
	\left| {\rm Im\,} \int_{\Omega_{\Psi}} (B_p \varphi) \bar{\varphi} |\varphi|^{p-2} \,d\nu \right|
	\le \ell_{N,p,\alpha} \left(
		{\rm Re\,} \int_{\Omega_{\Psi}} (B_p \varphi) \bar{\varphi} |\varphi|^{p-2}\,d\nu
			\right)
\end{align}%
holds for
$\varphi \in D(B_p)$.
Moreover, there exists a constant
$M_{N,p,\alpha} \ge 1$
depending only on $N, p$ and $\alpha$ such that
\begin{align}%
\label{1/t}
	\| B_p T(t) f \|_{L^p(\Omega_{\Psi}, d\nu)}
	\le \frac{M_{N,p,\alpha}}{t} \| f \|_{L^p(\Omega_{\Psi}, d\nu)}
\end{align}%
holds for
$f \in L^p(\Omega_{\Psi}, d\nu)$ and $t >0$.
\end{lemma}
%%%%%%%%%%%%%%%%%%%%%%%%%%%%%%%%%%%%%%%%%%
%%%%%%%%%%%%%%%%%%%%%%%%%%%%%%%%%%%%%%%%%%
\begin{proof}
Since
$B_{p^{\prime}}$
is the adjoint operator of
$B_p$
and
$\mathcal{D}_{\Psi}$
is a core for $B_p$ by Lemma \ref{lem_sg},
it suffices to prove \eqref{sect} for
$p \ge 2$
and
$\varphi \in \mathcal{D}_{\Psi}$.
Indeed, if \eqref{sect} is true for $p\ge 2$ and $\varphi \in \mathcal{D}_{\Psi}$,
we deduce for $p \le 2$ and $\varphi \in \mathcal{D}_{\Psi}$ that
\begin{align*}%
	\left| {\rm Im\,} \int_{\Omega_{\Psi}} (B_p \varphi) \bar{\varphi} |\varphi|^{p-2} \,d\nu \right|
	&=
	\left| {\rm Im\,} \int_{\Omega_{\Psi}} \varphi ( B_{p'} (\bar{\varphi} |\varphi|^{p-2}) )
		\,d\nu \right| \\
	&=
	\left| {\rm Im\,} \int_{\Omega_{\Psi}} \bar{f}|f|^{p-2} ( B_{p'} (f) )
		\,d\nu \right| \\
	&\le \ell_{N,p,\alpha} \left(
		{\rm Re\,} \int_{\Omega_{\Psi}} \bar{f} |f|^{p-2} (B_{p'} f) \,d\nu
			\right) \\
	&= \ell_{N,p,\alpha} \left(
		{\rm Re\,} \int_{\Omega_{\Psi}} \varphi (B_{p'} (\bar{\varphi} |\varphi|^{p-2})) \,d\nu
			\right) \\
	&=\ell_{N,p,\alpha} \left(
		{\rm Re\,} \int_{\Omega_{\Psi}} (B_p \varphi) \bar{\varphi} |\varphi|^{p-2}\,d\nu
			\right),
\end{align*}%
where $f = \bar{\varphi} |\varphi|^{p-2}$.
%Moreover, for $\varphi \in D(B_p)$, taking
%$\{ \varphi_j \}_{j=1}^{\infty} \subset \mathcal{D}_{\Psi}$ so that
%$\varphi_j \to \varphi $ in $D(B_p)$.... \label{保留}

Therefore, in what follows we assume
$p \ge 2$
and
$\varphi \in \mathcal{D}_{\Psi}$.
Setting $u = J \varphi \in \mathcal{D}$, we have 
\begin{align*}%
	\int_{\Omega_{\Psi}} ( B_p \varphi) \bar{\varphi} |\varphi|^{p-2}\,d\nu
	&= \int_{\Omega_{\Psi}} (J^{-1}Lu(y)) \overline{J^{-1}u(y)} |J^{-1}u(y)|^{p-2} \, d \nu\\
	&= \left( \frac{2}{2+\alpha} \right)^{(p-2)/2}
		\int_{\Omega} ( - \Delta u(x)) \overline{u(x)} |u(x)|^{p-2} |x|^{-\beta}\, dx,
\end{align*}%
where
$\beta = (N-2)\alpha(p-2)/4$.
Moreover, integration by parts yields
\begin{align}%
\label{Delta}
	\int_{\Omega} (-\Delta u ) \bar{u}|u|^{p-2} |x|^{-\beta}\, dx
	&= \int_{\Omega} ( \nabla u \cdot \nabla (\bar{u}|u|^{p-2} )) |x|^{-\beta}\, dx\\
\nonumber
	&\quad + \int_{\Omega} ( \nabla u \cdot \nabla ( |x|^{-\beta} ))
			\bar{u} |u|^{p-2}\,dx.
\end{align}%
By integration by parts again and taking the real-part, we see that
\begin{align*}%
	&{\rm Re\,}\int_{\Omega} ( \nabla u \cdot \nabla (\bar{u}|u|^{p-2} )) |x|^{-\beta}\, dx \\
	&\quad = {\rm Re\,}
	\int_{\Omega} \left( |\nabla u|^2 |u|^{p-2}
		+ (p-2) \nabla u \cdot ( \bar{u} |u|^{p-4} {\rm Re\,}(\bar{u}\nabla u) ) \right)
			|x|^{-\beta} \, dx \\
	&\quad = {\rm Re\,}
	\int_{\Omega} \left( | \bar{u}\nabla u |^2
		+ (p-2) (\bar{u} \nabla u) {\rm Re\,} (\bar{u} \nabla u) \right) |u|^{p-4} |x|^{-\beta}\,dx \\
	&\quad= (p-1) \int_{\Omega} | {\rm Re\,}(\bar{u}\nabla u) |^2 |u|^{p-4} |x|^{-\beta} \, dx
		+ \int_{\Omega} | {\rm Im\,}(\bar{u}\nabla u) |^2 |u|^{p-4} |x|^{-\beta}\,dx,\\
	&{\rm Re\,}\int_{\Omega} ( \nabla u \cdot \nabla ( |x|^{-\beta} ))
			\bar{u} |u|^{p-2}\,dx \\
	&\quad =
		\frac{1}{p} \int_{\Omega}
			\nabla (|u|^p) \cdot \nabla (|x|^{-\beta})\,dx \\
	&\quad =
		- \frac{1}{p} \int_{\Omega} |u|^p \Delta ( |x|^{-\beta} )\,dx \\
	&\quad =
		\frac{\beta(N-2-\beta)}{p} \int_{\Omega} |u|^p |x|^{-\beta-2}\,dx
\end{align*}%
and hence, we eventually have
\begin{align}%
\label{Re}
	{\rm Re\,} \int_{\Omega} (B_p \varphi) \bar{\varphi} |\varphi|^{p-2}\,d\nu
	&= (p-1)\left( \frac{2}{2+\alpha} \right)^{(p-2)/2}
	 \int_{\Omega} | {\rm Re\,}(\bar{u}\nabla u) |^2 |u|^{p-4} |x|^{-\beta} \, dx\\
\nonumber
	&\quad + \left( \frac{2}{2+\alpha} \right)^{(p-2)/2}
	\int_{\Omega} | {\rm Im\,}(\bar{u}\nabla u) |^2 |u|^{p-4} |x|^{-\beta}\,dx\\
\nonumber
	&\quad + \frac{\beta(N-2-\beta)}{p} \left( \frac{2}{2+\alpha} \right)^{(p-2)/2}
	\int_{\Omega} |u|^p|x|^{-\beta -2}\, dx.
\end{align}%
Therefore, $B_p$ is accretive if
$\beta \le N-2$, that is, $p \le 2 + 4/\alpha$.
On the other hand, taking the imaginary part of \eqref{Delta}, we have
\begin{align*}%
	\left| {\rm Im\,} \int_{\Omega} (B_p \varphi) \bar{\varphi} |\varphi|^{p-2}\,d\nu \right|
	&\le |p-2|\left( \frac{2}{2+\alpha} \right)^{(p-2)/2}
		 \int_{\Omega} |x|^{-\beta} | {\rm Re\,} ( \bar{u} \nabla u ) |
			| {\rm Im\,} (\bar{u} \nabla u ) | | u |^{p-4} \, dx\\
	&\quad + \beta \left( \frac{2}{2+\alpha} \right)^{(p-2)/2}
		\int_{\Omega} |x|^{-\beta -1}
			| {\rm Im\,}(\bar{u}\nabla u)| |u|^{p-2}\,dx.
\end{align*}%
Therefore, choosing
$2 \le p < 2+ 4/\alpha$ and
\begin{align*}%
	\ell_{N,p,\alpha}
	&= \sqrt{ \frac{(p-2)^2}{4(p-1)} + \frac{\beta p}{4(N-2-\beta)}} \\
	&= \sqrt{ \frac{(p-2)^2}{4(p-1)} + \frac{p(p-2)}{4}
		\left( 2 + \frac{4}{\alpha} -p \right)^{-1}},
\end{align*}%
we obtain
\begin{align*}%
	\left| {\rm Im\,} \int_{\Omega} (B_p \varphi) \bar{\varphi} |\varphi|^{p-2}\,d\nu \right|
	\le \ell_{N,p,\alpha} \left(
		{\rm Re\,} \int_{\Omega} (B_p \varphi) \bar{\varphi} |\varphi|^{p-2}\,d\nu \right).
\end{align*}%
This gives the $m$-sectoriality of $B_p$ and
hence $-B_p$ generates an analytic contraction semigroup on
$L^p(\Omega_{\Psi}, d\nu)$.
The estimate \eqref{1/t} is an immediate consequence of
a property of an analytic semigroup.
\end{proof}
%%%%%%%%%%%%%%%%%%%%%%%%%%%%%%%%%%%%%%%%%%

To deduce
$L^p$--$L^q$ estimates for
$T(t)$,
we use the following Gagliardo--Nirenberg inequality
(see for example \cite{Frbook}).
%%%%%%%%%%%%%%%%%%%%%%%%%%%%%%%%%%%%%%%%%%
\begin{lemma}\label{lem_gn3}
Let $N \ge 3$.
Then, there exists a constant $C_N > 0$ depending only on $N$ such that
for every $w \in W_0^{1,2}(\Omega_{\Psi})$, we have
\begin{align*}%
	\| w \|_{L^{2N/(N-2)}(\Omega_{\Psi}, d\nu)}
		\le C_N \| \nabla w \|_{L^2(\Omega_{\Psi}, d\nu)}.
\end{align*}%
\end{lemma}
%%%%%%%%%%%%%%%%%%%%%%%%%%%%%%%%%%%%%%%%%%%%
From this, we obtain the following estimate.
%%%%%%%%%%%%%%%%%%%%%%%%%%%%%%%%%%%%%%%%%%%%
\begin{lemma}\label{lem_embed}
Let $N \ge 3$ and let $2 \le p < 2+ 4/\alpha$.
Then, there exists a constant
$C_{N,p}>0$ depending only on $N$ and $p$ such that
for every $\varphi \in \mathcal{D}_{\Psi}$, we have
\begin{align}%
\label{bp_embed}
	\| \varphi \|_{L^{pN/(N-2)}(\Omega_{\Psi}, d\nu)}
		\le C_{N,p} \left(
			{\rm Re\,} \int_{\Omega_{\Psi}} (B_p \varphi) \bar{\varphi} |\varphi|^{p-2}\,d\nu
				\right)^{1/p}.
\end{align}%
\end{lemma}
%%%%%%%%%%%%%%%%%%%%%%%%%%%%%%%%%%%%%%%%%%%%
%%%%%%%%%%%%%%%%%%%%%%%%%%%%%%%%%%%%%%%%%%%%
\begin{proof}
Here we use the differential expression of $B$.
\begin{align*}%
	{\rm Re\,} \int_{\Omega_{\Psi}} (B_p \varphi) \bar{\varphi} |\varphi|^{p-2}\,d\nu
	&= {\rm Re\,} \int_{\Omega_{\Psi}}
			b(y) \nabla \varphi \cdot \nabla (\bar{\varphi} |\varphi|^{p-2} )\,dy \\
	&\quad - \frac{(N-2)^2\alpha (4+\alpha)}{16}
			\int_{\Omega_{\Psi}} \frac{ |\varphi(y)|^p }{|y|^2} \,dy\\
	&\ge (p-1) \left( \frac{2+\alpha}{2} \right)^2
		 \int_{\Omega_{\Psi}} |{\rm Re\,} (\bar{\varphi} \varphi_r)|^2 |\varphi|^{p-4}\,dy\\
	&\quad - \left( \frac{N-2}{2} \right)^2
			\left( \left( \frac{2+\alpha}{2} \right)^2 -1 \right)
			\int_{\Omega_{\Psi}} \frac{ |\varphi(y)|^p}{|y|^2}\,dy,
\end{align*}%
where
$\varphi_r$
stands for the derivative of $\varphi$ with respect to the radial direction.
The above inequality can be proved by recalling
$b(y) = I + \alpha( 1+ \frac{\alpha}{4} ) Q(\Psi^{-1}(y))$,
$Q(\Psi^{-1}(y) ) = ( \frac{y_i y_j}{|y|^2} )_{i,j=1,\ldots,N}$,
and noting that
\begin{align*}%
	&{\rm Re\,} \left[ b(y) \nabla \varphi \cdot \nabla ( \bar{\varphi} |\varphi|^{p-2} ) \right]\\
	&= |\varphi|^{p-4} \left[ (p-1)  | {\rm Re\,}(\bar{\varphi} \nabla \varphi) |^2
		+  | {\rm Im\,}(\bar{\varphi} \nabla \varphi) |^2 \right] \\
	&\quad + \alpha \left( 1+ \frac{\alpha}{4} \right) |\varphi|^{p-4} \left[
		(p-1)  \left| {\rm Re\,} \left(\frac{y}{|y|}\cdot (\bar{\varphi}\nabla \varphi) \right) \right|^2
		+  \left| {\rm Im\,} \left(\frac{y}{|y|}\cdot (\bar{\varphi}\nabla \varphi) \right) \right|^2
		\right] \\
	&\ge \left( 1+ \alpha \left( 1+ \frac{\alpha}{4} \right) \right)(p-1) |\varphi|^{p-4}
		\left| {\rm Re\,} \left(\frac{y}{|y|}\cdot (\bar{\varphi}\nabla \varphi) \right) \right|^2 \\
	&= (p-1) \left( \frac{2+\alpha}{2} \right)^2
		|{\rm Re\,} (\bar{\varphi} \varphi_r)|^2 |\varphi|^{p-4}
\end{align*}%
Applying the Hardy's inequality
\begin{align*}%
	 \int_{\Omega_{\Psi}} |{\rm Re\,}(\bar{\varphi} \varphi_r) |^2 |\varphi|^{p-4}\, dy
	\ge \left( \frac{N-2}{2} \right)^2 \frac{4}{p^2}
		\int_{\Omega_{\Psi}} \frac{ |\varphi(y)|^p}{|y|^2}\, dy
\end{align*}%
and
\begin{align*}%
	(p-1) \left( \frac{2+\alpha}{2} \right)^2 \left( \frac{N-2}{2} \right)^2 \frac{4}{p^2}
		- \left( \frac{N-2}{2} \right)^2
			\left( \left( \frac{2+\alpha}{2} \right) - 1 \right) > 0,
\end{align*}%
we see that there exists $\delta \in (0,1)$ such that
\begin{align*}%
	{\rm Re\,} \int_{\Omega_{\Psi}} ( B_p \varphi) \bar{\varphi} |\varphi|^{p-2}\,d\nu
		&\ge \delta^2 \int_{\Omega_{\Psi}}
			|\bar{\varphi} \varphi_r |^2 |\varphi|^{p-4}\,dy\\
		&= \frac{4\delta^2}{p^2}
			\int_{\Omega_{\Psi}} | \nabla |\varphi|^{p/2} |^2\, dx.
\end{align*}%
By Lemma \ref{lem_gn3} with
$w = |\varphi|^{p/2}$, we obtain
\begin{align*}%
	\| \varphi \|_{L^{pN/(N-2)}(\Omega_{\Psi}, d\nu)}
	&= \left\| | \varphi|^{p/2} \right\|_{L^{2N/(N-2)}(\Omega_{\Psi}, d\nu)}^{2/p} \\
	& \le C_N \left\| \nabla | \varphi |^{p/2} \right\|_{L^2(\Omega_{\Psi}, d\nu)}^{2/p}\\
	&\le C_N \left( \frac{2\delta}{p} \right)^{-2/p}
		\left( {\rm Re\,} \int_{\Omega_{\Psi}} ( B_p \varphi)
				\bar{\varphi} |\varphi|^{p-2}\,d\nu \right)^{1/p}.
\end{align*}%
Since $a_1 \le m$, the proof is finished.
\end{proof}
%%%%%%%%%%%%%%%%%%%%%%%%%%%%%%%%%%%%%%%%%%%%%
%%%%%%%%%%%%%%%%%%%%%%%%%%%%%%%%%%%%%%%%%%%%%
\begin{proof}[Proof of Proposition \ref{prop_lplq} for $N\ge 3$]
Let
$2 \le r < 1+4/\alpha$ and $f \in L^r(\Omega_{\Psi})$.
Applying \eqref{1/t} and Lemma \ref{lem_embed}, we have
\begin{align*}%
	\| T(t) f \|_{L^{rN/(N-2)}(\Omega_{\Psi}, d\nu)}
	&\le C_{N,r,\alpha} \left(
		{\rm Re\,}\int_{\Omega_{\Psi}}
			(B_r T(t) f ) \overline{T(t)f} |T(t)f|^{r-2} \,d\nu \right)^{1/r} \\
	&\le C_{N,r,\alpha} \| B_p T(t) f \|_{L^r(\Omega_{\Psi}, d\nu)}^{1/r}
		\| T(t) f \|_{L^r(\Omega_{\Psi}, d\nu)}^{1-1/r} \\
	&\le C_{N,r,\alpha} \left( M_{N,r,\alpha} \right)^{1/r}
		t^{-1/r} \| f \|_{L^r(\Omega_{\Psi}, d\nu)}.
\end{align*}%
%By interpolation inequality we deduce that if
%$2 \le r < 2+ 4/\alpha$ and $1/q \le 1/r \le N/(q(N-2))$, then
%\begin{align*}%
%	\| T(t) f \|_{L^q(\Omega_{\Psi}, d\nu)}.
%\end{align*}%
Fix
$2<p<(N/(N-2))(2+4/\alpha)$.
Then the case
$2<p \le 2N/(N-2)$ is already verified by the above result with
$r=2$ and $q=p$.

In the case $2N/(N-2) < p < (N/(N-2))(2+4/\alpha)$, we set
\begin{align*}%
	q_0 = 2,\quad
	q_m = \frac{N-2}{N}p,\quad
	q_j = \left( \frac{q_m}{q_0} \right)^{j/m} q_0\ (j=1,\ldots,m-1),
\end{align*}%
where $m$ is an integer satisfying
\begin{align*}%
	\left( 1 + \frac{2}{\alpha} \right)^{1/m} < \frac{N}{N-2}.
\end{align*}%
Since
\begin{align*}%
	2=q_0 < q_1 < \cdots < q_m < 2+\frac{4}{\alpha}
\end{align*}%
and
\begin{align*}%
	\frac{q_j}{q_{j-1}} = \left( \frac{q_m}{q_0} \right)^{1/m}
	= \left( \frac{N-2}{2N}p \right)^{1/m}
	< \left( 1+\frac{2}{\alpha} \right)^{1/m}
	< \frac{N}{N-2},
\end{align*}%
it follows from the previous discussion that
$L^{q_{j-1}}$-$L^{q_j}$ estimates are true for
$j=1,\ldots, m$ and
\begin{align*}%
	\| T(t/2) f \|_{L^{q_m}(\Omega_{\Psi}, d\nu)}
	&\le \left( \prod_{j=1}^m \| T(t/2m) \|_{L^{q_{j-1}}\to L^{q_j}} \right)
		\| f \|_{L^2(\Omega_{\Psi}, d\nu)}\\
	&\le C_{2, q_m} t ^{-(N/2)(1/2 - 1/q_m)} \| f \|_{L^2(\Omega_{\Psi}, d\nu)}.
\end{align*}%
Since we have
$\| T(t/2) \|_{L^{q_m} \to L^p} \le C_{q_m,p}t^{-1/p}$
in the previous step with
$r = (N-2)p/N$,
we obtain the desired assertion.
\end{proof}
%%%%%%%%%%%%%%%%%%%%%%%%%%%%%%%%%%%%%%%%%%%%%

%%%%%%%%%%%%%%%%%%%%%%%%%%%%%%%%%%%%%%%%%
%%%%%%%%%%%%%%%%%%%%%%%%%%%%%%%%%%%%%%%%%
%%%%%%%%%%%%%%%%%%%%%%%%%%%%%%%%%%%%%%%%%
%%%%%%%%%%%%%%%%%%%%%%%%%%%%%%%%%%%%%%%%%
\section{Related elliptic problem: construction of a weight function}
In this section, we construct a weight function,
which will be used for the weighted energy estimate of
the damped wave equation \eqref{dw}.
As we mentioned in the introduction,
in \cite[Remark 3.1]{SoWa16},
it is shown that the solution $A(x)$ of the Poisson equation
$\Delta A(x) = a(x)$
does not have the desired property in general
(in particular, the condition \eqref{a_est1} fails).
Therefore, following \cite{SoWa17},
we weaken the problem to the inequality
\begin{align}%
\label{ellip}
	(1-\varepsilon) a(x) \le \Delta A(x) \le (1+\varepsilon) a(x)
\end{align}%
for $x\in \Omega$, where $\varepsilon \in (0,1)$ is a parameter.
We construct a function $A_{\varepsilon}$
satisfying \eqref{ellip} and
\begin{align}%
\label{a_est1}
	&A_{1\varepsilon} \langle x \rangle^{2+\alpha}
		\le A_{\varepsilon}(x)
		\le A_{2\varepsilon} \langle x \rangle^{2+\alpha},\\
\label{a_est2}
	&\frac{|\nabla A_{\varepsilon}(x)|^2}{a(x)A(x)}
		\le \frac{2+\alpha}{N+\alpha} + \varepsilon
\end{align}%
for some $A_{1\varepsilon}, A_{2\varepsilon} > 0$.
%%%%%%%%%%%%%%%%%%%%%%%%%%%%%%%%%%%%%%%%%
\begin{lemma}\label{lem_ep}
For every $\varepsilon \in (0,1)$,
there exists
$A_{\varepsilon} \in C^2(\mathbb{R}^N)$
such that
\eqref{ellip}--\eqref{a_est2}.
\end{lemma}
%%%%%%%%%%%%%%%%%%%%%%%%%%%%%%%%%%%%%%%%%
%%%%%%%%%%%%%%%%%%%%%%%%%%%%%%%%%%%%%%%%%
\begin{proof}
While the proof is the same as that of \cite[Lemma 2.1]{SoWa17},
we give a proof for the reader's convenience.

First, noting $a(x) \in C^2(\bar{\Omega})$ and $a(x) > 0$,
we extend $a(x)$ as a positive function on the whole space $\mathbb{R}^N$
and denote it by $a(x)$ again.
We define
\begin{align*}%
	a_{\varepsilon}(x) := b_1(x) + \eta_{\varepsilon} b_2(x),
\end{align*}%
where
\begin{align*}%
	b_1(x) &= \Delta \left( \frac{a_0}{(N+\alpha)(2+\alpha)}\langle x \rangle^{2+\alpha} \right) \\
	&= a_0 \langle x \rangle^{\alpha}
		- \frac{a_0 \alpha}{N+\alpha} \langle x \rangle^{\alpha-2},
\end{align*}%
$b_2(x) = a(x) - b_1(x)$,
and
$\eta_{\varepsilon} \in C_c^{\infty}(\mathbb{R}^N, [0,1])$
is a cut-off function
such that $\eta_{\varepsilon} \equiv 1$ on $B(0,R_{\varepsilon})$
with some $R_{\varepsilon}>0$.
By the assumption \eqref{a},
for any $\varepsilon >0$,
we obtain
\begin{align}%
\label{a_ep}
	(1-\varepsilon) a(x) \le a_{\varepsilon} (x) \le (1+\varepsilon) a(x),
\end{align}%
provided that $R_{\varepsilon}$ is sufficiently large.

We take $A_{\varepsilon}$ as an appropriate solution of the Poisson equation
\begin{align}%
\label{ellip2}
	\Delta A_{\varepsilon}(x) = a_{\varepsilon}(x),
\end{align}%
that is, we define
\begin{align*}%
	A_{\varepsilon}(x)
	&:=
	\frac{a_0}{(N+\alpha)(2+\alpha)}\langle x \rangle^{2+\alpha}
	- \int_{\mathbb{R}^N} \mathcal{N}(x-y) \eta_{\varepsilon}(y) b_2(y)\,dy
	+ \lambda_{\varepsilon},
\end{align*}%
where
$\lambda_{\varepsilon}$ is a large constant determined later,
and
$\mathcal{N}$ is the Newton potential given by
\begin{align*}%
	\mathcal{N}(x) :=
		\left\{ \begin{array}{ll}
		\displaystyle \frac{1}{2\pi} \log \frac{1}{|x|}
		&\mbox{if}\ N=2,\\[8pt]
		\displaystyle \frac{\Gamma(N/2+1)}{N(N-2)\pi^{N/2}}|x|^{2-N}
		&\mbox{if}\ N\ge 3.
		\end{array}\right.
\end{align*}%
Then, clearly we have \eqref{ellip2} and
the property \eqref{ellip} is verified.
Since the leading term of $A_{\varepsilon}(x)$ for large $|x|$ is
$\frac{a_0}{(N+\alpha)(2+\alpha)}\langle x \rangle^{2+\alpha}$,
we have \eqref{a_est1},
provided that $\lambda_{\varepsilon} >0$ is sufficiently large.
Moreover, we easily compute
\begin{align*}%
	&\lim_{|x|\to \infty}
	\left( \frac{|\nabla A_{\varepsilon}(x)|^2}{a(x) A_{\varepsilon}(x)} \right) \\
	&\quad =
	\lim_{|x|\to\infty}
	\left( \frac{\langle x \rangle^{\alpha}}{a(x)}\cdot
		\frac{1}{\langle x \rangle^{-\alpha-2}A_{\varepsilon}(x)}
		\left| \frac{a_0}{N+\alpha} \frac{x}{\langle x \rangle}
			+ o(1) \right|^2 \right) \\
	&\quad = \frac{2+\alpha}{N+\alpha}.
\end{align*}%
Thus, retaking $\lambda_{\varepsilon}>0$ sufficiently large,
we have \eqref{a_est2}.
\end{proof}
%%%%%%%%%%%%%%%%%%%%%%%%%%%%%%%%%%%%%%%%%

%%%%%%%%%%%%%%%%%%%%%%%%%%%%%%%%%%%%%%%%%
%%%%%%%%%%%%%%%%%%%%%%%%%%%%%%%%%%%%%%%%%
%%%%%%%%%%%%%%%%%%%%%%%%%%%%%%%%%%%%%%%%%
%%%%%%%%%%%%%%%%%%%%%%%%%%%%%%%%%%%%%%%%%
\section{Weighted energy estimates for damped wave equations}

In this section, we give weighted energy estimates
for solutions to the damped wave equation \eqref{dw}.
As we described in the introduction,
we use the auxiliary function
$A_{\varepsilon}$ constructed in Section 3.

We first recall the finite propagation speed property for
the equation \eqref{dw}.
For the proof, see for example
\cite{Ikawa}.
%%%%%%%%%%%%%%%%%%%%%%%%%%%%%%%%%%%%%%
\begin{lemma}[Finite speed of the propagation]\label{lem_fps}
Let $u$ be a solution of the equation \eqref{dw} with initial data
$(u_0, u_1)$ satisfying
${\rm supp\,}(u_0, u_1) \subset \bar{\Omega} \cap B(0,R_0)$.
Then, we have
\begin{align*}%
	{\rm supp\,}u(\cdot,t) \subset
	\bar{\Omega} \cap B(0,R_0 +t)
\end{align*}%
for $t \ge 0$.
\end{lemma}
%%%%%%%%%%%%%%%%%%%%%%%%%%%%%%%%%%%%%%
%%%%%%%%%%%%%%%%%%%%%%%%%%%%%%%%%%%%%%
\begin{remark}
From Lemma \ref{lem_fps}, we see that
$\langle x\rangle \le R_0 +1+ t$
for $x \in {\rm supp\,}u(\cdot,t)$.
We shall frequently use this inequality.
\end{remark}
%%%%%%%%%%%%%%%%%%%%%%%%%%%%%%%%%%%%%%

We start with the following weighted energy identities.
In this step we do not need any special property
for the weight function.
\begin{lemma}[{\cite[Lemma 3.7]{SoWa16}}]
%%%%%%%%%%%%%%%%%%%%%%%%%%%%%%%%%%%%%%%%
\label{e0-0}
Let $\Phi\in C^2(\overline{\Omega}\times [0,\infty))$ 
satisfy $\Phi>0$ and $\pa_t\Phi<0$
and let $u$ be a solution of \eqref{dw}. Then, we have
\begin{align}
\label{en_eq1}
	&\frac{d}{dt}\left[ \int_{\Omega} \Big(|\nabla u|^2+|u_t|^2\Big) \Phi\,dx \right]\\
\nonumber
	&\quad = \int_{\Omega} (\pa_t\Phi)^{-1} \big| \pa_t\Phi\nabla u -u_t\nabla\Phi \big|^2 \,dx\\
\nonumber
	&\qquad + \int_{\Omega}
		\Big( -2a(x)\Phi+\pa_t\Phi - (\pa_t\Phi)^{-1}|\nabla\Phi|^2 \Big)|u_t|^2 \,dx.   
\end{align}
\end{lemma}%%%%%%%%%%%%%%%%%%%%%%%%%%%%%%%%%%%%%%%%%%%%%%%%

\begin{lemma}[{\cite[Lemma 3.9]{SoWa16}}]
%%%%%%%%%%%%%%%%%%%%%%%%%%%%%%%%%%%%%%%%%%%%%%
\label{e1-0}
Let $\Phi\in C^2(\overline{\Omega}\times [0,\infty))$ 
satisfy $\Phi>0$ and $\pa_t\Phi<0$
and let $u$ be a solution to \eqref{dw}.
Then, we have
\begin{align}
\label{en_eq2}
	\frac{d}{dt} \left[ \int_{\Omega} \Big(2uu_t+a(x)|u|^2\Big) \Phi\,dx  \right]
	&= 2\int_{\Omega} uu_{t} (\pa_t\Phi)\,dx
	+ 2\int_{\Omega} |u_t|^2 \Phi\,dx  
	- 2\int_{\Omega} |\nabla u|^2 \Phi\,dx \\
\nonumber
	&\quad + \int_{\Omega} \big(a(x)\pa_t\Phi+\Delta\Phi\big)|u|^2\,dx.
\end{align}
\end{lemma}%%%%%%%%%%%%%%%%%%%%%%%%%%%%%%%%%%%%%%%%%%%%%%%

Using the function $A_{\varepsilon}(x)$ constructed
in the previous section,
we introduce our weight function.
%%%%%%%%%%%%%%%%%%%%%%%%%%%%%%%%%%%%%%%%%%%%%%%%%
\begin{definition}
Let $h:=\frac{2+\alpha}{N+\alpha}$ 
and $\ep\in (0,1)$.
We define
\begin{equation}
\Phi_{\ep}(x,t)
=
\exp\left(
\frac{1}{h+2\ep}\,\frac{A_\ep(x)}{1+t}
\right),
\end{equation}
where $A_{\ep}$ is given in Lemma \ref{lem_ep}. 
For $t \ge 0$, we also define the following energy
\begin{gather}
\label{en_xta}
E_{\pa x}(t;u)
:=
\int_{\Omega}|\nabla u|^2\Phi_\ep\,dx,
\quad
E_{\pa t}(t;u)
:=
\int_{\Omega}|u_t|^2\Phi_\ep\,dx,
\\
\label{en_*A}
E_{a}(t;u)
:=
\int_{\Omega}a(x)|u|^2\Phi_\ep\,dx,
\quad
E_{*}(t;u)
:=
2\int_{\Omega}uu_t\Phi_\ep\,dx,
\end{gather}
and
$E_1(t;u):=E_{\pa x}(t;u)+E_{\pa t}(t;u)$ 
and $E_2(t;u):=E_{*}(t;u)+E_{a}(t;u)$.
\end{definition}
%%%%%%%%%%%%%%%%%%%%%%%%%%%%%%%%%%%%%%%%%%%%%%%

The main result of this section is the following.
We give weighted energy estimates for solutions 
of \eqref{dw}.

%%%%%%%%%%%%%%%%%%%%%%%%%%%%%%%%%%%%%%%%%%%%%%%%%
\begin{proposition}\label{main}
Assume that $(u_0,u_1)$ satisfies 
${\rm supp}\,(u_0,u_1)\subset \overline{B}(0,R_0)$ 
and the compatibility condition of 
order $k_0\geq 1$. 
Let $u$ be a solution of the problem \eqref{dw}. 
For every $\delta>0$ and $0 \leq k\leq k_0-1$, 
there exist $\ep>0$ and $M_{\delta,k,R_0}>0$ such that 
%$\lambda (\ep)=\frac{1-5\ep}{(1+\ep)(h+2\ep)}$ 
%$(=h^{-1}-\delta_\ep)$
for every $t\geq 0$, 
\begin{align*}
&(1+t)^{\frac{N+\alpha}{2+\alpha}+2k+1-\delta}
\Big(E_{\pa x}(t;\pa_t^{k}u)+E_{\pa t}(t;\pa_t^{k}u)\Big)
+
(1+t)^{\frac{N+\alpha}{2+\alpha}+2k-\delta}E_a(t;\pa_t^{k}u) \\
&\quad \leq M_{\delta,k,R_0}\|(u_0,u_1)\|_{H^{k+1}\times H^k(\Omega)}^2.
\end{align*}
\end{proposition}
%%%%%%%%%%%%%%%%%%%%%%%%%%%%%%%%%%%%%%%%%%%%%%

To prove this,
we need to prepare the following lemmas.
First, we calculate derivatives of the weight function $\Phi_{\ep}$. 
%%%%%%%%%%%%%%%%%%%%%%%%%%%%%%%%%%%%%%%%%%%%%%%%%%%%
\begin{lemma}\label{lem_phi}
We have
\begin{align*}%
	\pa_t \Phi_{\ep}(x,t) &= -\frac{1}{h+2\ep} \frac{A_{\ep}(x)}{(1+t)^2} \Phi_{\ep}(x,t),\\
	\nabla \Phi_{\ep}(x,t) &=
		\frac{1}{h+2\ep} \frac{\nabla A_{\ep}(x)}{1+t} \Phi_{\ep}(x,t),\\
	\Delta \Phi_{\ep}(x,t) &=
		\frac{1}{h+2\ep} \frac{\Delta A_{\ep}(x)}{1+t} \Phi_{\ep}(x,t)
			+ \left| \frac{1}{h+2\ep} \frac{\nabla A_{\ep}(x)}{1+t} \right|^2
			\Phi_{\ep}(x,t).
\end{align*}%
In particular, we obtain
\begin{align*}%
	-\Delta \Phi_{\ep}(x,t) + \frac{|\nabla \Phi_{\ep}(x,t)|^2}{\Phi_{\ep}(x,t)}
		&= - \frac{1}{h+2\ep} \frac{\Delta A_{\ep}(x)}{1+t} \Phi_{\ep}(x,t),\\
	 \frac{|\nabla \Phi_{\ep}(x,t)|^2}{\pa_t \Phi_{\ep}(x,t)}
	 	&= - \frac{1}{h+2\ep} \frac{|\nabla A_{\ep}(x)|^2}{A_{\ep}(x)} \Phi_{\ep}(x,t).
\end{align*}%
\end{lemma}
%%%%%%%%%%%%%%%%%%%%%%%%%%%%%%%%%%%%%%%%%%%%%%%%%%%%%%
The proof is straightforward and we omit it.
The second tool is a Hardy type inequality.
\begin{lemma}%%%%%%%%%%%%%%%%%%%%%%%%%%%%%%%%%%%%%%%%%%%%%%
\label{lem_ha}
For $t\geq 0$, we have
\begin{align}
\label{hardy}
\frac{1-\ep}{h+2\ep}\,\frac{1}{1+t}
E_{a}(t;u)
	\leq  E_{\pa x}(t;u).
\end{align}
\end{lemma}
%%%%%%%%%%%%%%%%%%%%%%%%%%%%%%%%%%%%%%%%
%%%%%%%%%%%%%%%%%%%%%%%%%%%%%%%%%%%%%%%
\begin{proof}
As in the proof of \cite[Lemma 3.6]{SoWa16}, 
by integration by parts we have
\begin{align*}
	\int_{\Omega}
	\Delta(\log\Phi_\ep)|u|^2\Phi_{\ep}\,dx
	&=
	\int_{\Omega}\left(\Delta\Phi_\ep-\frac{|\nabla\Phi_\ep|^2}{\Phi_\ep}\right)|u|^2\,dx \\
	&= \int_{\Omega}|\nabla u|^2\,\Phi_\ep\,dx
		- \int_{\Omega} \Phi_{\ep}^{-1} | \nabla ( \Phi_{\ep} u ) |^2\, dx\\
	&\leq 
	\int_{\Omega}|\nabla u|^2\,\Phi_\ep\,dx.
\end{align*}
Using Lemma \ref{lem_phi} and \eqref{ellip}, we see that
\[
\Delta(\log\Phi_\ep(x))=\frac{1}{h+2\ep}\,\frac{\Delta A_{\ep}(x)}{1+t}
\geq 
\frac{1-\ep }{h+2\ep}\,\frac{a(x)}{1+t},
\]
which leads to \eqref{hardy}.
\end{proof}
%%%%%%%%%%%%%%%%%%%%%%%%%%%%%%%%%%%%%%%%%%%%%%%

Next, in order to clarify the effect of the finite propagation speed property, 
we put
\[
a_1:=\inf_{x\in \Omega}\Big(\lr{x}^{-\alpha} a(x)\Big).
\]
Then we have the following.
%%%%%%%%%%%%%%%%%%%%%%%%%%%%%%%%%%%%%%%%%%%
\begin{lemma}\label{lem_appfps}
For $t\geq 0$, we have
\begin{align}
\label{E12-F}
E_{\pa t}(t;u)
&\leq \frac{1}{a_1} E_{a}(t;\pa_t u),
\\
\label{A/a}
\int_{\Omega} \frac{A_\ep(x)}{a(x)}|u_{t}|^2 \Phi_\ep \,dx
&\leq \frac{A_{2\ep}}{a_1}(R_0+1+t)^{2}
E_{\pa t}(t;u),
\\
\label{E21-F}
|E_{*}(t;u)|
&\leq 
\frac{2}{\sqrt{a_1}}
\sqrt{E_a(t;u)E_{\pa t}(t;u)}. 
\end{align}
\end{lemma}
%%%%%%%%%%%%%%%%%%%%%%%%%%%%%%%%%%%%%%%%%%%%%%%

\begin{proof}
By $a(x)/ a_1 \ge \lr{x}^{\alpha} \ge 1$,
we have 
\begin{align*}
\int_{\Omega}|u_t|^2\Phi_\ep\,dx
\leq \int_{\Omega}\frac{a(x)}{a_1}|u_t|^2\Phi_\ep\,dx
\leq \frac{1}{a_1} E_a(t;\pa_t u),
\end{align*}
which shows \eqref{E12-F}.
By \eqref{a_est1},
$1/a(x) \leq (a_1 \lr{x}^{\alpha} )^{-1}$
and Lemma \ref{lem_fps}, we have
\begin{align*}%
	\frac{A_{\ep}(x)}{a(x)}
	\le \frac{A_{2\ep} \lr{x}^{2+\alpha}}{a_1 \lr{x}^{\alpha}}
	\le \frac{A_{2\ep} }{a_1} (R+t+1)^2,
\end{align*}%
which implies \eqref{A/a}.
Finally, using the Cauchy-Schwarz inequality and
the inequality $a(x)/ a_1 \ge \lr{x}^{\alpha} \ge 1$,
we see that
\begin{align*}
\left|\int_{\Omega}uu_t\Phi_\ep\,dx\right|^2
&\leq 
\frac{1}{a_1}\left(\int_{\Omega}a(x)|u|^2\Phi_\ep\,dx\right)
E_{\pa t}(t;u)
\\
&\leq 
\frac{1}{a_1}
E_a(t;u)E_{\pa t}(t;u),
\end{align*}
which yields \eqref{E21-F}.
The proof of \eqref{E21-F} is similar and we omit it.
\end{proof}

Applying lemmas proved above,
we have the following differential inequality
for $E_1(t;u)$ and $E_2(t;u)$.
%%%%%%%%%%%%%%%%%%%%%%%%%%%%%%%%%%%%%%%%%%%%%
\begin{lemma}
{\bf (i)}\ For every $t\geq 0$, we have
\begin{equation}\label{e0-1}
%	\frac{d}{dt}\left[ \int_{\Omega} \Big(|\nabla u|^2+|u_t|^2\Big) \Phi_\ep\,dx \right]
    \frac{d}{dt}E_1(t;u)
	\leq 
    -E_{a}(t;\pa_t u).
    %-F_{a}(t;\pa_t u).
%	-\int_{\Omega}
%	\left(a(x)+\frac{A_{\ep}(x)}{(h+2\ep)(1+t)^2}\right)
%	|u_t|^2\Phi_\ep
%	\,dx.
\end{equation}
{\bf (ii)}\ For every $\ep\in (0,1/3)$ and $t\geq 0$, we have
\begin{align}
%	\frac{d}{dt} \left[ \int_{\Omega} \Big(2uu_t+a(x)|u|^2\Big) \Phi_\ep\,dx  \right]
    \frac{d}{dt}E_{2}(t;u)
	&\leq 
	-\frac{1-3\ep}{1-\ep}E_{\pa x}(t;u)
\label{e1-1}
	\\
\nonumber
	&\quad +
	\left(\frac{2}{a_1}+\frac{A_{2\ep}(R_0+1)^2}{\ep a_1^2}\right)E_{a}(t;\pa_t u).
\end{align}
\end{lemma}
%%%%%%%%%%%%%%%%%%%%%%%%%%%%%%%%%%%%%%%%%%%

%%%%%%%%%%%%%%%%%%%%%%%%%%%%%%%%%%%%%%%%%%%
\begin{proof}
To prove \eqref{e0-1}, we use Lemma \ref{e0-0}
and estimate the last term of \eqref{en_eq1}.
By virtue of Lemma \ref{lem_phi}, \eqref{a_est2} and that $A_{\ep} \ge 0$,
we have
\begin{align*}
&-2a(x)\Phi_\ep+\pa_t\Phi_\ep - (\pa_t\Phi_\ep)^{-1}
|\nabla \Phi_\ep|^2\\
&\quad=
\left(-2a(x)-\frac{A_{\ep}(x)}{(h+2\ep)(1+t)^2}
+\frac{1}{h+2\ep}\frac{|\nabla A_\ep(x)|^2}{A_\ep(x)}\right)\Phi_\ep
\\
&\quad\leq 
\left(
-2a(x)
%-\frac{A_{\ep}(x)}{(h+2\ep)(1+t)^2}
+\frac{h+\ep}{h+2\ep}a(x)\right)\Phi_\ep
\\
&\quad\leq 
-
a(x)\Phi_\ep,
\end{align*}
which implies \eqref{e0-1}.
Next, we use Lemma \ref{e1-0} to verify \eqref{e1-1}
and estimate the last term of \eqref{en_eq2}.
It follows from Lemma \ref{lem_phi} and \eqref{ellip} that
\begin{align*}
a(x)\pa_t\Phi_\ep+\Delta\Phi_\ep
&=
\frac{1}{h+2\ep}
\left(-\frac{a(x)A_{\ep}(x)}{(1+t)^2}
+\frac{|\nabla A_\ep(x)|^2}{(h+2\ep)(1+t)^2}
+\frac{\Delta A_{\ep}(x)}{1+t}
\right)\Phi_\ep
\\
&
\leq
\left(
-\frac{\ep}{(h+2\ep)^2}\,\frac{a(x)A_{\ep}(x)}{(1+t)^2}
+\frac{1+\ep}{h+2\ep}\,\frac{a(x)}{1+t}
\right)
\Phi_\ep. 
\end{align*}
and hence, with Lemma \ref{lem_ha}, 
we have 
\begin{align*}
&\int_{\Omega} \big(a(x)\pa_t\Phi_\ep+\Delta\Phi_\ep\big)|u|^2\,dx\\
&\quad \leq
\frac{1+\ep}{1-\ep}\int_{\Omega} |\nabla u|^2\Phi_\ep\,dx
-\frac{\ep}{(h+2\ep)^2}\,
\frac{1}{(1+t)^2}
\int_{\Omega} a(x)A_\ep(x)|u|^2\Phi_\ep\,dx.
\end{align*}
By using \eqref{A/a} and \eqref{E12-F},
the first term of the right-hand side of \eqref{en_eq2}
is also estimated as
\begin{align*}
&2\int_{\Omega} uu_{t} (\pa_t\Phi_\ep)\,dx \\
&\quad =
-\frac{2}{h+2\ep}\frac{1}{(1+t)^2}\int_{\Omega} uu_{t} A_\ep(x)\Phi_\ep\,dx
\\
&\quad \leq
\frac{2}{h+2\ep}\frac{1}{(1+t)^2}
\left(\int_{\Omega} a(x)A_\ep(x)|u|^2\Phi_\ep\,dx\right)^\frac{1}{2}
\left(\int_{\Omega} \frac{A_\ep(x)}{a(x)}|u_{t}|^2 \Phi_\ep \,dx\right)^\frac{1}{2}
\\
&\quad \leq
\frac{\ep}{(h+2\ep)^2}\,\frac{1}{(1+t)^2}
\int_{\Omega} a(x)A_\ep(x)|u|^2\Phi_\ep\,dx
+
\frac{A_{2\ep}(R_0+1)^2}{\ep a_1^2}
E_a(t; \pa_tu).
\end{align*}
Combining them with \eqref{en_eq2},
we reach \eqref{e1-1}.
\end{proof}

Furthermore,
multiplying
$E_1(t;u)$ and $E_2(t;u)$
by the time weight functions
$(t_1+t)^m$ and $(t_2 + t)^{\lambda}$,
respectively,
we have the following time-weighted differential inequalities.

%%%%%%%%%%%%%%%%%%%%%%%%%%%%%%%%%%%%%%%%
\begin{lemma}\label{est}
The following assertions hold:

\smallskip

\noindent{\bf (i)} %
Set 
$t_*(\alpha,m):= 2m/a_1$. 
Then for every $t,m\geq 0$ and $t_1\geq t_*(\alpha,m)$,  
\begin{align}
\label{e1-m}
\frac{d}{dt}
\Big((t_1+t)^{m}E_1(t;u)\Big)
\leq 
m
(t_1+t)^{m-1}E_{\pa x}(t;u)
- \frac{1}{2}(t_1+t)^{m}E_a(t;\pa_t u).
\end{align}
\noindent{\bf (ii)} %
For every $t,\lambda\geq 0$ and $t_2\geq 1$,  
\begin{align}
\label{e2-l}
&\frac{d}{dt}\Big((t_2+t)^{\lambda}E_2(t;u)\Big) \\
\nonumber
&\quad\leq 
\lambda(1+\ep)(t_2+t)^{\lambda-1}
E_a(t;u)
-\frac{1-3\ep}{1-\ep}
(t_2+t)^{\lambda}E_{\pa x}(t;u)
\\
\nonumber
&\qquad
+
\left(
\frac{2}{a_1}+\frac{A_{2\ep}(R_0+1)^2}{\ep a_1^2}+
\frac{\lambda}{2\ep a_1^2 t_2}
\right)(t_2+t)^{\lambda}E_a(t;\pa_t u).
\end{align}
\noindent{\bf (iii)} %
In particular, setting 
\begin{align*}
\nu
&:=
\frac{4}{a_1}+\frac{2A_{2\ep}(R_0+1)^2}{\ep a_1^2}+
\frac{1}{\ep a_1}, 
\\
t_{**}(\ep,\alpha,\lambda)
&:=
\max\left\{
	\frac{(1-\ep) \lambda \nu}{\ep},
	\lambda, 1, t_{\ast}(\alpha,m)
	\right\},
\end{align*}
one has that for $t,\lambda\geq 0$ and 
$t_3\geq t_{**}(\ep,\alpha,\lambda)$,  
\begin{align}
\nonumber
&\frac{d}{dt}
\Big(
\nu (t_3+t)^{\lambda}E_1(t;u)
+(t_3+t)^{\lambda}E_2(t;u)
\Big)
\\
\label{e3-l}
&\quad \leq 
-\frac{1-4\ep}{1-\ep}
(t_3+t)^{\lambda}E_{\pa x}(t;u)
+
\lambda(1+\ep)(t_3+t)^{\lambda-1}
E_a(t;u).
\end{align}
\end{lemma}

\begin{proof}
{\bf (i)} 
Let $m\geq 0$ and $t_1\geq t_{*}(\alpha,m)$. 
Using \eqref{e0-1} and \eqref{E12-F}, we have 
\begin{align*}
&(t_1+t)^{-m}
\frac{d}{dt}
\Big((t_1+t)^{m}E_1(t;u)\Big) \\
&\quad =
\frac{m}{t_1+t}
E_{\pa x}(t;u) 
+
\frac{m}{t_1+t}
E_{\pa t}(t;u)
+\frac{d}{dt}E_1(t;u)
\\
&\quad \leq 
\frac{m}{t_1+t}
E_{\pa x}(t;u) 
+
\frac{m}{t_1+t}
E_{\pa t}(t;u)-E_a(t;\pa_t u)
\\
&\quad \leq 
\frac{m}{t_1+t}
E_{\pa x}(t;u) 
+ 
\left(\frac{m}{a_1(t_1+t)} - 1\right) E_a(t;\pa_t u).
\end{align*}
Since
$t_{\ast} = 2m/a_1$,
we see that
\begin{align*}%
	\frac{m}{a_1(t_1+t)} - 1 \le - \frac12
\end{align*}%
holds for $t_1 \ge t_{\ast}$ and $t\ge 0$.
This implies \eqref{e1-m}.

\noindent{\bf (ii)} %%(ii)
By using \eqref{e2-l} and \eqref{e1-1},
$t\geq 0$ and $t_2 \geq 1$, we have
\begin{align*}%
&(t_2+t)^{-\lambda}
\frac{d}{dt}\Big((t_2+t)^{\lambda}E_2(t;u)\Big)
\\
&\quad \leq 
\frac{\lambda}{t_2+t}E_*(t;u)
+ 
\frac{\lambda}{t_2+t}E_a(t;u)+\frac{d}{dt}E_2(t;u)
\\	
&\quad \leq 
\frac{\lambda}{t_2+t}E_*(t;u)
+ 
\frac{\lambda}{t_2+t}E_a(t;u)
-\frac{1-3\ep}{1-\ep}
E_{\pa x}(t;u) \\
&\qquad +\left(\frac{2}{a_1}+\frac{A_{2\ep}(R_0+1)^2}{\ep a_1^2}\right)
E_a(t;\pa_t u).
\end{align*}%
Noting that \eqref{E21-F} and \eqref{E12-F} lead to
\begin{align*}%
\frac{\lambda}{t_2+t}E_{*}(t;u)
&\leq 
\frac{2\lambda}{\sqrt{a_1}(t_2+t)}
\sqrt{ E_a(t;u) E_{\pa t}(t;u) }
\\
&\leq
\frac{2\lambda}{a_1(t_2+t)}
\sqrt{ E_a(t;u) E_{a}(t; \pa_t u) }
\\
&\leq 
\frac{\lambda\ep}{t_2+t}
E_a(t;u)
+
\frac{\lambda}{\ep a_1^2 (t_2+t)}
E_a(t;\pa_t u),
\end{align*}%
we deduce \eqref{e2-l}.

\smallskip 

\noindent{\bf (iii)}  
Combining \eqref{e1-m} with $m=\lambda$ 
and \eqref{e2-l}, we have
for $t_3\geq t_{**}(\ep,\alpha,\lambda)$ and $t\geq 0$, 
\begin{align*}%
&\frac{d}{dt}
\Big(
\nu(t_3+t)^{\lambda}E_1(t;u)
+(t_3+t)^{\lambda}E_2(t;u)
\Big)
\\
&\leq 
\left(
\frac{\nu \lambda}{ t_3+t }
-\frac{1-3\ep}{1-\ep}
\right)(t_3+t)^{\lambda}E_{\pa x}(t;u)
+\lambda(1+\ep)(t_3+t)^{\lambda-1}
E_a(t;u)
\\
&\quad
+\left(
\frac{2}{a_1}+\frac{A_{2\ep}(R_0+1)^2}{\ep a_1^2}+
\frac{\lambda}{2\ep a_1^2 t_3}
- \frac{\nu}{2}\right)
(t_3+t)^{\lambda}E_a(t;\pa_t u)
\\
&\leq 
-
\frac{1-4\ep}{1-\ep}
(t_3+t)^{\lambda}E_{\pa x}(t;u)
+\lambda(1+\ep)(t_3+t)^{\lambda-1}
E_a(t;u),
\end{align*}%
which gives the assertion.
\end{proof}

%%%%%%%%%%%%%%%%%%%%%%%%%%%%%%%%%%
\begin{proof}[Proof of Proposition \ref{main}]
By \eqref{E21-F}, we easily have 
\begin{align}
\label{ea}
\nu E_1(t;u)+E_2(t;u)&\geq \frac{3}{4}E_a(t;u).
\end{align}
By using the above estimate, 
we prove the assertion via mathematical induction. 

\noindent
{\bf Step 1 ($k=0$).}\ 
From \eqref{e3-l} 
and Lemma \ref{lem_ha}, we deduce that
\begin{align*}
&\frac{d}{dt}
\Big(
\nu (t_3+t)^{\lambda}E_1(t;u)
+(t_3+t)^{\lambda}E_2(t;u)
\Big) \\
&\quad \leq 
\left(
-\frac{1-4\ep}{1-\ep}
+
\frac{\lambda(1+\ep)(h+2\ep)}{1-\ep}
\right)
(t_3+t)^{\lambda}E_{\pa x}(t;u).
\end{align*}
Taking 
$\lambda_0=\frac{(1-\ep)(1-4\ep)}{(1+\ep)(h+2\ep)}$,
integrating over $(0,t)$ with respect to $t$
and using \eqref{ea},
we have
\begin{align*}
&\frac{3}{4}(t_3+t)^{\lambda_0}E_a(t;u)
+
\frac{\ep(1-4\ep)}{1-\ep}
\int_{0}^t(t_3+s)^{\lambda_0}E_{\pa x}(s;u)\,ds
\\
&
\leq 
\nu t_3^{\lambda_0}E_1(0;u)
+t_3^{\lambda_0}E_2(0;u).
\end{align*}
This gives the desired estimate for $E_a(t;u)$.
Also, as a byproduct, we obtain
\begin{align*}%
	\frac{\ep(1-4\ep)}{1-\ep}
	\int_{0}^t(t_3+s)^{\lambda_0}E_{\pa x}(s;u)\,ds
	&\le \nu t_3^{\lambda_0}E_1(0;u)
+t_3^{\lambda_0}E_2(0;u).
\end{align*}%
Using this inequality and \eqref{e1-m} with $m=\lambda_0+1$, 
and integrating over $(0,t)$, we obtain
\begin{align*}
&(t_3+t)^{\lambda_0+1}E_1(t;u)
+\frac{1}{2}
\int_{0}^t(t_3+s)^{\lambda_0+1}E_a(s;\pa_t u)\,ds
\\
&\quad \leq 
t_3^{\lambda_0+1}E_1(0;u)
+
(\lambda_0+1)
\int_0^t
(t_3+s)^{\lambda_0}E_{\pa x}(s;u)\,ds
\\
&\quad \leq
t_3^{\lambda_0+1}E_1(0;u)+\frac{(\lambda_0+1)(1-\ep)}
{\ep(1-4\ep)}
\Big(
\nu t_3^{\lambda_0}E_1(0;u)
+t_3^{\lambda_0}E_2(0;u)
\Big).
\end{align*}
Finally, putting $\delta = 1/h - \lambda_0$,
we have the desired assertion with $k=0$.
Moreover, we also have
\begin{align*}%
	&\frac{1}{2}
	\int_{0}^t(t_3+s)^{\lambda_0+1}E_a(s;\pa_t u)\,ds \\
	&\le t_3^{\lambda_0+1}E_1(0;u)+\frac{(\lambda_0+1)(1-\ep)}
	{\ep(1-4\ep)}
	\Big(
	\nu t_3^{\lambda_0}E_1(0;u)
	+t_3^{\lambda_0}E_2(0;u)
	\Big),
\end{align*}%
which will be used in the next step.
\smallskip

\noindent{\bf Step 2 ($1<k\leq k_0-1$).}\ 
Suppose that 
for every $t\geq 0$, 
\[
(1+t)^{\lambda_0+2k-1}
E_{1}(t;\pa_t^{k-1}u)
+
(1+t)^{\lambda_0+2k-2}E_a(t;\pa_t^{k-1}u)
\leq M_{\ep,k-1}\|(u_0,u_1)\|_{H^{k}\times H^{k-1}(\Omega)}^2
\]
and additionally, 
\[
\int_{0}^t(1+s)^{\lambda_0+2k-1}E_a(s;\pa_t^{k}u)\,ds
\leq M_{\ep, k-1}'\|(u_0,u_1)\|_{H^{k}\times H^{k-1}(\Omega)}^2
\]
are valid.
Since the initial value $(u_0,u_1)$ 
satisfies the compatibility condition
of order $k$, 
$\pa_t^{k} u$ is also a solution of 
\eqref{dw} with the initial data
$(\pa_t^{k}u, \pa_t^{k+1}u)(x,0) = (u_{k-1},u_{k})(x)$.
Applying \eqref{e3-l} with $\lambda=\lambda_0+2k$,  
putting $t_{3k}=t_{**}(\ep,R_0,\alpha,\lambda_0+2k)$ 
(see Lemma \ref{est}\ {\bf (iii)}),
integrating over $(0,t)$,
and noting \eqref{ea},
we have
\begin{align*}
&
\frac{3}{4}
(t_{3k}+t)^{\lambda_0+2k}E_a(t;\pa_t^{k} u)
+\frac{1-4\ep}{1-\ep}
\int_0^t(t_{3k}+s)^{\lambda_0+2k}E_{\pa x}(s;\pa_t^{k} u)\,ds
\\
&\quad \leq 
\nu t_{3k}^{\lambda_0+2k}E_1(0;\pa_t^{k} u)
+t_{3k}^{\lambda_0+2k-1}E_2(0;\pa_t^{k} u) \\
&\qquad +
(\lambda_0+2k)(1+\ep)
M_{\ep, k-1}'\|(u_0,u_1)\|_{H^{k}\times H^{k-1}(\Omega)}^2.
\end{align*}
In particular, we obtain the boundedness of
the second term of the left-hand side.
From this and \eqref{e1-m} with $m=\lambda_0+2k+1$
we conclude 
\begin{align*}
&
(t_{3k}+t)^{\lambda_0+2k+1}E_1(t;\pa_t^{k}u)
+
\frac{1}{2}
\int_0^t(t_{3k}+s)^{\lambda_0+2k+1}E_a(s;\pa_t^{k+1}u)
\,ds
\\
&\quad \leq M''_{\ep,k}
\Big(
E_1(0;\pa_t^{k} u)+E_2(0;\pa_t^{k} u)+\|(u_0,u_1)\|_{H^{k}\times H^{k-1}(\Omega)}^2
\Big)
\end{align*}
with some constant $M''_{\ep,k}>0$.
Therefore, by induction, we obtain the desired inequalities for all $k\leq k_0-1$. 
\end{proof}

%%%%%%%%%%%%%%%%%%%%%%%%%%%%%%%%%%%%%%%%%
%%%%%%%%%%%%%%%%%%%%%%%%%%%%%%%%%%%%%%%%%
%%%%%%%%%%%%%%%%%%%%%%%%%%%%%%%%%%%%%%%%%
%%%%%%%%%%%%%%%%%%%%%%%%%%%%%%%%%%%%%%%%%
\section{Proof of the diffusion phenomena}
In this section, we give a proof of Theorem \ref{thm1}.
We use the following lemma stated in 
\cite[Section 4]{SoWa16}. 
\begin{lemma}\label{Duhamel}
Assume that 
$
	(u_0,u_1) \in (H^2\cap H^1_0(\Omega))
		\times H^1_0(\Omega)
$
and
suppose that ${\rm supp}\,(u_0,u_1)\subset \{ x\in \Omega ; |x| \le R_0 \}$. 
Then for every $t\geq 0$, 
\begin{align}
\nonumber
	u(x,t) - e^{-tL_{\ast}}[ u_0 + a(\cdot)^{-1}u_1]
	 &= - \int_{t/2}^t e^{-(t-s)L_{\ast}}[ a(\cdot)^{-1}u_{tt}(\cdot, s) ] ds\\
\nonumber
	&\quad -e^{-\frac{t}{2}L_{\ast}} [ a(\cdot)^{-1}u_{t}(\cdot, t/2) ] \\
\label{eq_du2}
	&\quad +\int_0^{t/2} L_{\ast}  e^{-(t-s)L_{\ast}}[ a(\cdot)^{-1}u_{t}(\cdot, s) ] ds,
\end{align}
where
$L_{*}$
is the Friedrichs extension of
$L=-a(x)^{-1}\Delta$ in $L^2_{d\mu}$.
\end{lemma}

\begin{proof}[Proof of Theorem \ref{thm1}]
First we show the assertion for $(u_0,u_1)$ 
satisfying the compatibility condition of order $2$. 
Taking $L^2_{d\mu}$-norm of both side, we have
\[
\Big\|u(x,\cdot) - e^{-tL_{\ast}}[ u_0 + a(\cdot)^{-1}u_1]\Big\|_{L^2_{d\mu}}
\leq 
\mathcal{J}_1(t)
+\mathcal{J}_2(t)
+\mathcal{J}_3(t),
\]
where
\begin{align*}
\mathcal{J}_1(t)&:=
\int_{t/2}^t 
\big\|e^{-(t-s)L_{\ast}}[ a(\cdot)^{-1}u_{tt}(\cdot, s)]\big\|_{L^2_{d\mu}} ds, 
\\
\mathcal{J}_2(t)&:=\big\|e^{-\frac{t}{2}L_{\ast}} [ a(\cdot)^{-1}u_{t}(\cdot, t/2)]\big\|_{L^2_{d\mu}},
\\
\mathcal{J}_3(t)&:=\int_0^{t/2} \big\|L_{\ast}  e^{-(t-s)L_{\ast}}[ a(\cdot)^{-1}u_{t}(\cdot, s)]\big\|_{L^2_{d\mu}} ds.
\end{align*}
Noting that
$a(x) \ge a_1$ and $\Phi(x,t) \ge 1$,
and applying Proposition \ref{main} with $k=1$ and $k=0$
we see that
\begin{align*}
\big\|a(\cdot)^{-1}\pa_t^{k+1}u(\cdot, s)\big\|_{L^2_{d\mu}}^2
&=
  \int_{\Omega}a(x)^{-1}|\pa_t^{k+1}u(\cdot, s)|^2\,dx
\\
&\leq 
	\frac{1}{a_1} \int_{\Omega}|\pa_t^{k+1}u(\cdot, s)|^2\Phi_\ep\,dx
\\
&\leq 
  \frac{1}{a_1} E_{\pa t}(t,\pa_t^{k}u)
\\
&\leq 
	\frac{M_{\delta,k}}{a_1} (1+t)^{-\frac{N+\alpha}{2+\alpha}-2k-1+\delta}
  \|(u_0,u_1)\|_{H^{k+1}\times H^k}^2.
\end{align*}
Therefore, we have
\begin{align*}
\mathcal{J}_1(t)
&\leq 
\int_{t/2}^t 
\big\|a(\cdot)^{-1}u_{tt}(\cdot, s)\big\|_{L^2_{d\mu}} ds
\\
&\leq 
  \sqrt{\frac{M_{\delta,1}}{a_1}}
  \|(u_0,u_1)\|_{H^{2}\times H^1}
  \int_{t/2}^t (1+s)^{-\frac{N+\alpha}{2(2+\alpha)}-\frac32+\frac{\delta}{2}} ds
\\
&\leq 
  C_{1,\delta}
  (1+t)^{-\frac{N+\alpha}{2(2+\alpha)}-\frac12+\frac{\delta}{2}}
  \|(u_0,u_1)\|_{H^{2}\times H^1},
\end{align*}
where $C_{1,\delta}$ is a positive constant depending on $\delta$.
In the same way, we deduce
\begin{align*}
\mathcal{J}_2(t)
&\leq \big\|a(\cdot)^{-1}u_{t}(\cdot, t/2)\big\|_{L^2_{d\mu}}
\leq 
  \sqrt{\frac{M_{\delta,0}}{a_1}}
  	(1+t)^{-\frac{N+\alpha}{2(2+\alpha)}-\frac12 + \frac{\delta}{2}}
\|(u_0,u_1)\|_{H^{1}\times L^2}.
\end{align*}
Next, we estimate the term $\mathcal{J}_3(t)$.
%%%%%%%%%%%%%%%%%%%%%%%%%%%%%% J_3(t) for N=2.
First, we treat the case $N = 2$.
In this case, by Proposition \ref{prop_lplq2}, 
we have
\begin{align*}
	\mathcal{J}_3(t)
	&\leq 
	C\int_0^{t/2} 
	(t-s)^{-\frac32}
	\big
	\|a(\cdot)^{-1}u_{t}(\cdot, s)\big\|_{L^1_{d\mu}} ds \\
	&\leq 
	C t^{-\frac32}
	\int_0^{t/2} 
	\sqrt{\|\Phi_\ep^{-1}(\cdot,s)\|_{L^1(\Omega)}E_{\pa t}(s;u)}\,
	ds.
\end{align*}
We compute
\begin{align*}
  \|\Phi_{\ep}^{-1}(\cdot,t)\|_{L^1(\Omega)}
  &\leq 
  \int_{\R^2}
  	\exp\left(-\frac{A_{1\ep}}{h+2\ep}\,\frac{|x|^{2+\alpha}}{1+t} \right)\,dx
  \\
  &=
  (1+t)^{\frac{2}{2+\alpha}}
  \int_{\R^2}\exp\left(-\frac{A_{1\ep}}{h+2\ep}|y|^{2+\alpha}\right)\,dy.
\end{align*}
From this and Proposition \ref{main}, we conclude
\begin{align*}%
	\mathcal{J}_3(t)
	&\leq
	C t^{-\frac32}
	\int_0^{t/2}
		(1+s)^{\frac{1}{2+\alpha}-\frac{2+\alpha}{2(2+\alpha)}-\frac12+\frac{\delta}{2}} ds
		\cdot \| (u_0,u_1) \|_{H^1 \times L^2}\\
	&\leq C t^{-\frac{2+\alpha}{2(2+\alpha)} - \frac{1+\alpha}{2+\alpha} + \frac{\delta}{2}}
		\| (u_0,u_1) \|_{H^1 \times L^2}.
\end{align*}%

%%%%%%%%%%%%%%%%%%%%%%%%%% J_3 for N \ge 3.
When $N\ge 3$, Proposition \ref{prop_lplq2} leads to
\begin{align*}
	\mathcal{J}_3(t)
	&\leq 
	C\int_0^{t/2} 
	(t-s)^{-\frac{N}{2}\left(\frac{1}{q}-\frac{1}{2}\right)-1}
		\left(
		\int_{\Omega} | a(x)^{-1}u_t(x,s)|^q |x|^{(N-2)\alpha (2-q)/4} d\mu(x)
		\right)^{1/q}
	ds \\
	&\leq 
	C
	t^{-\frac{N}{2}\left(\frac{1}{q}-\frac{1}{2}\right)-1}
	\int_0^{t/2} 
		\left(
		\int_{\Omega} |u_t(x,s)|^q a(x)^{-(q-1)} |x|^{(N-2)\alpha (2-q)/4} dx
		\right)^{1/q}
	ds,
\end{align*}
where
$q \in ((p_{\alpha})^{\prime}, 2]$
and sufficiently close to $(p_{\alpha})^{\prime}$.
Using $a(x) \ge a_1 \langle x \rangle^{\alpha}$
and the Cauchy--Schwarz inequality,
we further calculate
\begin{align*}%
	J_3(t)
	&\le C t^{-\frac{N}{2}\left(\frac{1}{q}-\frac{1}{2}\right)-1}
		\int_{0}^{t/2}
		\left\| u_t(s) |x|^{%
			\alpha \left[ \frac{N}{2} \left( \frac{1}{q} - \frac{1}{2} \right)-\frac12 \right]}
		\right\|_{L^q} ds\\
	&\le C t^{-\frac{N}{2}\left(\frac{1}{q}-\frac{1}{2}\right)-1}
		\int_{0}^{t/2}
		\left\| u_t(s) \Phi_{\ep}^{1/2} \right\|_{L^2}
		\left\| \Phi_{\ep}^{-1/2}  |x|^{%
			\alpha \left[ \frac{N}{2} \left( \frac{1}{q} - \frac{1}{2} \right)-\frac12 \right]}
			\right\|_{L^{2q/(2-q)}} ds.
\end{align*}%
We also notice that
if $q$ is close to $(p_{\alpha})^{\prime}$,
then
$\frac{N}{2} \left( \frac{1}{q} - \frac{1}{2} \right)-\frac12 > 0$
holds.
Therefore,
by changing variables, we estimate
\begin{align*}%
	\left\| \Phi_{\ep}^{-1/2}  |x|^{%
			\alpha \left[ \frac{N}{2} \left( \frac{1}{q} - \frac{1}{2} \right)-\frac12 \right]}
			\right\|_{L^{2q/(2-q)}}
	&\le C (1+t)^{%
		\frac{\alpha}{2+\alpha}%
		\left[ \frac{N}{2} \left( \frac{1}{q} - \frac{1}{2} \right)-\frac12 \right]%
		+ \frac{N}{2+\alpha} \left( \frac{1}{q} - \frac{1}{2} \right)} \\
	&\le C (1+t)^{%
		\frac{N}{2}\left(\frac{1}{q}-\frac{1}{2}\right)-\frac{\alpha}{2(2+\alpha)}}.
\end{align*}%
From this estimate and Proposition \ref{main}, we obtain
\begin{align*}%
	J_3(t)&\le
	C(1+t)^{-\frac{N}{2}\left(\frac{1}{q}-\frac{1}{2}\right)-1}
		\int_{0}^{t/2} (1+s)^{%
			\frac{N}{2}\left(\frac{1}{q}-\frac{1}{2}\right)-\frac{\alpha}{2(2+\alpha)}}
		(1+s)^{-\frac{N+\alpha}{2(2+\alpha)}-\frac{1}{2}+\frac{\delta}{2}}
		ds\cdot
		\| (u_0,u_1) \|_{H^1 \times L^2}\\
	&\le C
	(1+t)^{-\frac{N}{2}\left(\frac{1}{q}-\frac{1}{2}\right)-1%
		+\max\left\{ %
			0, \frac{N}{2}\left(\frac{1}{q}-\frac{1}{2}\right)%
				-\frac{N+2\alpha}{2(2+\alpha)}-\frac{1}{2}+\frac{\delta}{2}%
		\right\} + \delta^{\prime}}
		\| (u_0,u_1) \|_{H^1 \times L^2}
\end{align*}%
with arbitrary small $\delta^{\prime} > 0$.
Taking $q$ sufficiently close to $(p_{\alpha})^{\prime}$,
we conclude
\begin{align*}%
	J_3(t)&\le
	C (1+t)^{-\frac{N+\alpha}{2(2+\alpha)}%
		- \frac{1+\alpha}{2+\alpha} + \delta^{\prime \prime}}
		\| (u_0,u_1) \|_{H^1 \times L^2},
\end{align*}%
where
$\delta^{\prime\prime}>0$
depends on
$\delta$, $\delta^{\prime}$
and it can be taken arbitrary small.

%%%%%%%%%%%%%%%%%%%%%%%%
Finally, combining the estimates for
$\mathcal{J}_1(t), \mathcal{J}_2(t), \mathcal{J}_3(t)$
and noting
$1/2 < (1+\alpha)/(2+\alpha)$,
we have the desired estimate.

Next we show the assertion for $(u_0,u_1)$ 
satisfying $(u_0,u_1)\in (H^2\times H^1_0(\Omega))\times H^1_0(\Omega)$
(the compatibility condition of order $1$) 
via an approximation argument. 
Fix $\phi\in C_c^\infty(\R^N,[0,1])$ such that $\phi\equiv1$ on $\overline{B}(0,R_0)$ and $\phi\equiv0$ on $\R^N\setminus B(0,R_0+1)$ and define 
for $n\in\N$, 
\[
\left(\begin{array}{c}
u_{0n}
\\
u_{1n}
\end{array}\right)
=
\left(\begin{array}{c}
\phi \tilde{u}_{0n}
\\
\phi \tilde{u}_{1n}
\end{array}\right),
\quad 
\left(\begin{array}{c}
\tilde{u}_{0n}
\\
\tilde{u}_{1n}
\end{array}\right)
=
\left(1+\frac{1}{n}\mathcal{A}\right)^{-1}
\left(\begin{array}{c}
u_0
\\
u_1
\end{array}\right),
\]
where $\mathcal{A}$ is an quasi-$m$-accretive operator 
in $\mathcal{H}=H^1_0(\Omega)\times L^2(\Omega)$ 
associated with \eqref{dw}, that is, 
\[
\mathcal{A}=
\left(\begin{array}{cc}
0 & -1
\\
-\Delta & a(x)
\end{array}\right)
\]
endowed with domain 
$D(\mathcal{A})=(H^2\cap H^1_0(\Omega))\times H_0^1(\Omega)$. 
Then $(u_{0n},u_{1n})$ satisfies ${\rm supp} (u_{0n},u_{1n})\subset \overline{B}(0,R_0+1)$
and the compatibility condition of order $2$. 
Let $v_n$ be a solution of \eqref{dw} with 
$(u_{0n},u_{1n})$. 
Observe that 
\begin{align*}
\|(u_{0n},u_{1n})\|_{H^2\times H^1}^2
&\leq C^2\|\phi\|_{W^{2,\infty}}^2
\|(\tilde{u}_{0},\tilde{u}_{1})\|_{H^2\times H_1}^2
\\
&\leq C'^2\|\phi\|_{W^{2,\infty}}^2
(\|(\tilde{u}_{0},\tilde{u}_{1})\|_{\mathcal{H}}^2
+\|\mathcal{A}(\tilde{u}_{0},\tilde{u}_{1})\|_{\mathcal{H}}^2)
\\
&\leq C'^2\|\phi\|_{W^{2,\infty}}^2
(\|(u_{0},u_1)\|_{\mathcal{H}}^2
+\|\mathcal{A}(u_{0},u_{1})\|_{\mathcal{H}}^2)
\\
&\leq C''^2\|\phi\|_{W^{2,\infty}}^2
\|(u_{0},u_{1})\|_{H^2\times H^1}^2
\end{align*}
with suitable constants $C$, $C'$, $C''>0$, 
and 
\begin{gather*}
\left(\begin{array}{c}
u_{0n}
\\
u_{1n}
\end{array}\right)
\to 
\left(\begin{array}{c}
\phi u_0
\\
\phi u_1
\end{array}\right)
=
\left(\begin{array}{c}
u_0
\\
u_1
\end{array}\right)
\quad 
\text{in}\ \mathcal{H}
\end{gather*}
as $n\to \infty$ and also $u_{0n}+a^{-1}u_{1n}\to u_0+a^{-1}u_1$ 
in $L^2_{d\mu}$ as $n\to \infty$. 
Using the result of the previous step, we deduce
\[
\Big\|v_n(\cdot,t) - e^{tL_{\ast}}[ u_{0n} + a(\cdot)^{-1}u_{1n}]\Big\|_{L^2_{d\mu}}
\leq 
\tilde{C}(1+t)^{-\frac{N+\alpha}{2(2+\alpha)}-\frac{1+\alpha}{2+\alpha}+\delta}
\|(u_0,u_1)\|_{H^{2}\times H^1}
\]
with some constant $\tilde{C}>0$. 
Letting $n\to\infty$, by continuity of the $C_0$-semigroup $e^{-t\mathcal{A}}$ 
in $\mathcal{H}$
we also obtain diffusion phenomena 
for initial data in $(H^2\cap H^1_0(\Omega))\cap H^1_0(\Omega)$. 
\end{proof}

\section*{Acknowledgments}

This work is partially supported 
by Grant-in-Aid for Young Scientists Research (B), 
No.\ 16K17619
and Grant-in-Aid for Young Scientists Research (B), 
No.\ 16K17625.

\end{document}